\documentclass[11pt]{amsart} \usepackage{latexsym, amssymb, amsthm, times, verbatim, amsmath, stmaryrd}
\usepackage[T1]{fontenc}

\newtheorem{thm}{Theorem}[section]
\newtheorem{lemma}[thm]{Lemma}
\newtheorem{prop}[thm]{Proposition}
\newtheorem{cor}[thm]{Corollary}
\newtheorem{question}[thm]{Question}
\newtheorem{fact}[thm]{Fact}

\theoremstyle{definition}
\newtheorem{df}[thm]{Definition}
\newtheorem{nrmk}[thm]{Remark}

\newtheorem{notation}[thm]{Notation}
\newtheorem{ex}[thm]{Example}

\theoremstyle{remark}

\renewcommand{\r}{\mathbb{R}}
\newcommand{\Z}{\mathbb{Z}}

\newcommand{\curly}[1]{\mathcal{#1}}

\newcommand{\B}{\curly{B}}

\newcommand{\F}{\curly{F}}

\newcommand{\n}{\mathbb{N}}

\renewcommand{\to}{\rightarrow}

\newcommand\concat{\widehat{\phantom{\eta}}}
\def \balpha{\boldsymbol\alpha}
\def \bbeta{\boldsymbol\beta}

\def \<{\langle}
\def \>{\rangle}

\def \*Z {{{^*}\Z}}
\def \((  {(\!(}
\def \)) {)\!)}

\def \m{\operatorname{M}}

\def \tp{\operatorname{tp}}

\numberwithin{equation}{section}

\def \tr{\operatorname{tr}}

\def \Th{\operatorname{Th}}
\def \R{\mathcal R}
\def \u{\mathcal U}

\def \m{\mathcal{M}}
\def \good{\operatorname{good}}

\def  \Fix{\operatorname{Fix}}
\def \H{\mathcal{H}}
\def \U{\operatorname{U}}
\def \vNa{\operatorname{vNa}}

\makeatletter

\def\indsym#1#2{%
  \setbox0=\hbox{$\m@th#1x$}%
  \kern\wd0%
  \hbox to 0pt{\hss$\m@th#1\mid$\hbox to 0pt{$\m@th#1^{#2}$}\hss}%
  \lower.9\ht0\hbox to 0pt{\hss$\m@th#1\smile$\hss}%
  \kern\wd0}

\def\dotminussym#1#2{%
  \setbox0=\hbox{$\m@th#1-$}%
  \kern.5\wd0%
  \hbox to 0pt{\hss\hbox{$\m@th#1-$}\hss}%
  \raise.6\ht0\hbox to 0pt{\hss$\m@th#1.$\hss}%
  \kern.5\wd0}
\newcommand{\dotminus}{\mathbin{\mathpalette\dotminussym{}}}

\def\nindsym#1#2{%
  \setbox0=\hbox{$\m@th#1x$}%
  \kern\wd0%
  \hbox to 0pt{\hss$\m@th#1\not$\kern1.4\wd0\hss}
  \hbox to 0pt{\hss$\m@th#1\mid$\hbox to 0pt{$\m@th#1^{\,#2}$}\hss}%
  \lower.9\ht0\hbox to 0pt{\hss$\m@th#1\smile$\hss}%
  \kern\wd0}

\allowdisplaybreaks[2]

\DeclareMathOperator{\Aut}{Aut}

\DeclareMathOperator{\SL}{SL}

\newcommand{\cstar}{$\mathrm{C}^*$}
\def \Met{\operatorname{Met}}
\def \Mod{\operatorname{Mod}}
\def \Erg{\operatorname{Erg}}

\title{Spectral gap and definability}
\author{Isaac Goldbring}
\thanks{Goldbring's work was partially supported by NSF CAREER grant DMS-1349399.}

\address {Department of Mathematics, University of California, Irvine, 340 Rowland Hall (Bldg.\# 400), Irvine, CA, 92697-3875.}
\email{isaac@math.uci.edu}
\urladdr{http://www.math.uci.edu/~isaac}

\begin{document}

\begin{abstract}
We relate the notions of spectral gap for unitary representations and subfactors with definability of certain important sets in the corresponding structures.  We give several applications of this relationship.
\end{abstract}

\maketitle

\tableofcontents

\section{Introduction}

The notion of definable set is one of (if not \emph{the}) most important concepts in classical model theory.  Indeed, if one wants to understand the model-theoretic properties of a given structure and/or make use of this understanding in applications, a thorough analysis of the definable sets in the structure is often indispensable.

While continuous model theory resembles classical model theory in many respects and a general pattern of formulating continuous analogues of classical definitions and results soon becomes apparent, the continuous logic notion of definable set sometimes can pose a problem for classical model theorists.  Since continuous logic formulae are real-valued, a na\"ive first guess is that definable sets in continuous logic should just be level sets of formulae, that is, for a given structure $\mathbf{M}$, formula $\varphi(\vec x)$ (perhaps with parameters in $\mathbf{M}$), and $r\in \r$, one might guess that $\{\vec a\in \mathbf{M} \ : \ \varphi^{\mathbf{M}}(\vec a)=r\}$ should be a definable set.  While it is true that ultimately all definable sets are of this form, for various reasons one quickly realizes that this na\"ive notion of definable set is not robust enough to carry out many familiar arguments using definable sets.  For instance, one would like to be able to take a formula and quantify some of the free variables over a definable set and be left with a formula again.  In general, this property does not hold when quantifying over level sets.

Eventually, a notion of definable set in continuous logic was proposed that allowed for the remainder of the basic theory to go through as in the classical case.  While the definition served the correct theoretical purpose, some practical criticisms remained, namely to give an analysis of the definable sets and functions in some basic metric structures.  This proves to be a much more difficult endeavor than in classical logic.  We made an attempt at such an analysis in the papers \cite{Gold1}, \cite{Gold2}, and \cite{Gold3}, but our efforts were far from yielding a complete classification.  

In this paper, we switch our perspective and instead show how a particularly important notion in von Neumann algebras, namely that of \textbf{spectral gap}, is intimately related to the definability of certain naturally occurring sets.  Spectral gap is an integral part of Popa's deep theory of \textbf{deformation rigidity} in the study of II$_1$ factors; see, for example, his ICM survey \cite{popaICM}.  Our analysis in this paper merely scratches the surface of the model-theoretic study of spectral gap and it is our hopes that a much finer analysis could shed some light on the underlying model theory behind Popa's extremely successful programme.

We would be remiss if we did not point out that the notion of definability has played an extremely important role in the model-theoretic study of \cstar-algebras.  An extensive treatment of this topic can be found in \cite{munster}.

We now sketch the contents of our paper.  First, at the end of this introduction, we give an extremely rapid introduction to continuous logic; by no means is this introduction exhaustive but rather the reader should view it as our attempt to fix notation and presentation.  In Section 2, we give a careful treatment of the notion of definable set.  We take this opportunity to offer some alternative nomenclature and motivation that might prove psychologically useful to classical model theorists trying to understand the rationale behind the continuous logic definition of definable set.

In Section 3, we begin the study of the connection between spectral gap and definability in the context of unitary group representations, in which this notion originally came to light.  This allows us the opportunity to give a model-theoretic treatment of this notion in a context that is much more simple and involves far less prerequisites.

In Section 4, we introduce the basic facts that we need about von Neumann algebras for the rest of the paper.  Once again, our treatment is quick and a more detailed presentation of the subject aimed at model theorists can be found in \cite{Gold4}.  

In Section 5, we arrive at the main results in the paper.  We give a model-theoretic treatment of spectral gap subalgebra and use it to give some applications, including a technologically more elementary proof of the fact that the theory of tracial von Neumann algebras does not have a model companion, which was originally proven in \cite{nomodcomp} with Bradd Hart and Thomas Sinclair.

In the final section, we take the opportunity to discuss our paper \cite{braddisaachenry}, joint with Bradd Hart and Henry Towsner.  In that paper, we give a model-theoretic account of the striking paper \cite{BCI}, where they show that the family of continuum many pairwise non-isomoprhic separable II$_1$ factors constructed by McDuff are in fact pairwise non-elementarily equivalent.  It turns out that in the background of both papers, heavy uses of spectral gap are employed and it is the aim of the final section to spell out these hidden uses in more detail.

We would like to thank Bradd Hart, Thomas Sinclair, and Todor Tsankov for many useful conversations regarding this work.  We would also like to thank the Institut Henri Poincar\'e for their incredible hospitality during our stay there, where the majority of this paper was written.

\subsection{A crash course in continuous logic}

In this subsection, we give a very short introduction to continuous logic.  We make no attempt at being exhaustive, but rather use this as a chance to fix our terminology, notation and setup.  We borrow heavily from \cite{munster}.  

\begin{df}
A \textbf{metric structure} is a triple $\mathbf{M}:=(\mathbf{S}(\mathbf{M}),\mathbf{F}(\mathbf{M}),\mathbf{R}(\mathbf{M}))$ where:
\begin{enumerate}
\item $\mathbf{S}(\mathbf{M})$ is an indexed family of complete bounded metric spaces, called the sorts of $\mathbf{M}$;
\item $\mathbf{F}(\mathbf{M})$, the set of distinguished functions, is a set of uniformly continuous functions such that, for $f\in \mathbf{F}(\mathbf{M})$, the domain of $f$ is a finite product of sorts and the codomain of $f$ is another sort.
\item $\mathbf{R}(\mathbf{M})$, the set of distinguished relations, is a set of uniformly continuous functions such that, for $R\in \mathbf{R}(\mathbf{M})$, the domain of $R$ is a finite product of sorts and the codomain of $R$ is a compact interval in $\r$.
\end{enumerate}
\end{df}

As in classical model theory, one uses a \textbf{signature} to describe a metric structure.  Formally, a signature is a triple $L:=(\mathfrak{S},\mathfrak{F},\mathfrak{R})$, where:
\begin{itemize}
\item $\mathfrak{S}$ is a set of sorts.  Moreover, for each sort $S\in \mathfrak{S}$, the signature must provide a real number $K_S$.
\item For each function symbol $F\in \mathfrak{F}$, the signature must provide an arity, which is a finite sequence $(S_1,\ldots,S_n)$ from $\mathfrak{S}$ (the domain of $F$) together with another element $S$ from $\mathfrak{S}$ (the codomain of $F$).  In addition, the signature must provide a \textbf{modulus of uniform continuity} for $F$.
\item For each relation symbol $R\in \mathfrak{R}$, the signature must provide an arity, which is a finite sequence $(S_1,\ldots,S_n)$ from $\mathfrak{S}$ (the domain of $R$) together with a compact interval $K_R$ in $\r$ (the codomain of $R$).  Once again, the signature must provide a modulus of uniform continuity.
\end{itemize}

Given a signature $L$, there is a natural notion of a metric structure $\mathbf{M}$ being an $L$-structure.  The key points are as follows:
\begin{itemize}
\item $\mathbf{S}(\mathbf{M})$ is indexed by $\mathfrak{S}$.  Moreover, for each $S\in \mathfrak{S}$, letting $(S(\mathbf{M}),d_{S(\mathbf{M})})$ denote the metric space indexed by $S$, $K_S$ is a bound on the diameter of $(S(\mathbf{M}),d_{S(\mathbf{M})})$.
\item $\mathbf{F}(\mathbf{M})$ is indexed by $\mathfrak{F}$.  Moreover, for each $F\in \mathfrak{F}$, letting $F^{\mathbf{M}}$ denote the corresponding distinguished function, we have that the domain, codomain, and modulus of uniform continuity of $F^{\mathbf{M}}$ are as prescribed by the signature. 
\item The analogous statement as in the previous item for $\mathfrak{R}$.
\end{itemize}

Given a metric signature $L$, one defines \textbf{$L$-terms} as in classical logic.  \textbf{Atomic $L$-formulae} are given by:
\begin{itemize}
\item $Rt_1\cdots t_n$, where $R\in \mathfrak{R}$, the domain of $R$ is $(S_1,\ldots,S_n)$, and each $t_i$ is a term of sort $S_i$;
\item $d_S(t_1,t_2)$, where $t_1$ and $t_2$ are terms of sort $S$. 
\end{itemize} 
One obtains arbitrary $L$-formulae by closing under all continuous functions $\r^n\to \r$ and the \textbf{quantifiers} $\sup$ and $\inf$.

Given an $L$-formula $\varphi(\vec x)$, an $L$-structure $\textbf{M}$, there is a natural notion of the \textbf{interpretation} of $\varphi$ in $\bf{M}$, which is a function $\varphi^{\bf{M}}:{\bf{M}}^{\vec x}\to \r$.  Each $\varphi^{\bf{M}}$ is a uniformly continuous function taking values in a compact interval $K_\varphi$ in $\r$ in a way that depends only on $L$ (and not on the choice of $\bf{M}$).

An \textbf{$L$-sentence} is an $L$-formula with no free variables.  An \textbf{$L$-theory} is a set of $L$-sentences.  Given an $L$-theory $T$ and an $L$-structure $\bf{M}$, we say that \textbf{$\bf{M}$ is a model of $T$}, denoted $\bf{M}\models T$, if $\sigma^{\bf M}=0$ for all $\sigma \in T$.  We let $\Mod(T)$ denote the class of all models of $T$.  A class $\mathcal{C}$ of $L$-structures is called an \textbf{elementary class} if there is an $L$-theory $T$ such that $\mathcal{C}=\Mod(T)$.

As in classical logic, there are appropriate notions of \textbf{embedding}, \textbf{elementary embedding}, \textbf{substructure}, and \textbf{elementary substructure}.

Ultraproducts will play an important role in this paper, so let us recall the construction in continuous logic.  Let $(\mathbf{M}_i)_{i\in I}$ denote a family of $L$-structures and let $\u$ be an ultrafilter on $I$.  For each sort $S$, we let $d_{S,\u}$ denote the pseudometric on $\prod_{i\in I}S(\mathbf{M}_i)$ given by $d_{S,\u}(x_i,y_i):=\lim_{\u}d_{S(\mathbf{M}_i)}(x_i,y_i)$.  The \textbf{ultraproduct} of the family $(\mathbf{M}_i)$, denoted $\prod_\u \mathbf{M}_i$ has as the underlying metric space of sort $S$ the quotient of $\prod_{i\in I}S(\mathbf{M}_i)$ by the pseudometric $d_{S,\u}$ and with symbols interpreted coordinatewise as usual.\footnote{Taken literally, it seems that a relation symbol $R$ with values in the interval $K_R$ now take values in an ultrapower $K_R^\u$ of $K_R$.  However, $K_R^\u$ is naturally isomorphic to $K_R$ via the ultralimit map.}  When $\mathbf{M}_i=\mathbf{M}$ for all $i$, we speak instead of the \textbf{ultrapower} of $\mathbf{M}$, denoted $\mathbf{M}^\u$.  For $(x_i)\in \prod_{i\in I} S(\mathbf{M}_i)$, we let $(x_i)^\bullet$ denote its equivalence class in the ultraproduct.   The diagonal embedding of $\mathbf{M}$ in $\mathbf{M}^\u$ is the embedding given by mapping $a\in \mathbf{M}$ to $(a,a,a,\ldots)^\bullet$.  The \L os theorem for continuous logic states that, for an arbitrary formula $\varphi(\vec x)$ and $\vec a=(\vec a_i)^\bullet\in (\prod_\u \mathbf{M}_i)^{\vec x}$, we have that $$\varphi^{\prod_\u \mathbf{M}_i}(\vec a):=\lim_\u \varphi^{\mathbf{M}_i}(\vec a_i).$$  It follows that a diagonal embedding is always an elementary map.

In Section 5, existentially closed structures will be a point of discussion.  For the sake of brevity, we give the semantic definition.  If $\mathbf{M}$ and $\mathbf{N}$ are $L$-structures with $\mathbf{N}\subseteq \mathbf{M}$, we say that $\mathbf{N}$ is \textbf{existentially closed} (or e.c.) in $\mathbf{M}$ if there is an embedding $\mathbf{M}\hookrightarrow \mathbf{N}^{\u}$ such that the restriction to $\mathbf{N}$ is the diagonal embedding.  If $T$ is an $L$-theory, we call a model $\mathbf{N}\models T$ an existentially closed model of $T$ if it is existentially closed in all extensions that are also models of $T$.

\section{Definability in continuous logic}

\subsection{Generalities on formulae}

Until further notice, we fix a base theory $T$.  The goal of this subsection is to extend the notion of formula in models of $T$.  Towards this end, we let $\F_{\vec x}^0$ denote the set of $L$-formulae with free variables $\vec x$.  Observe that $\F_{\vec x}^0$ is naturally a pseudometric space when equipped with the pseudometric
$$d(\varphi,\psi):=d_{\F_{\vec x}^0}(\varphi,\psi):=\sup\{|\varphi^{\mathbf{M}}(\vec a)-\psi^{\mathbf{M}}(\vec a)| \ : \ \mathbf{M}\models T, \ \vec a\in \mathbf{M}^{\vec x}\}.$$

$d$ is indeed a pseuodmetric rather than a metric as two formulae $\varphi$ and $\psi$ are \textbf{$T$-equivalent} precisely when $d(\varphi,\psi)=0$.

When given a pseudometric space $(X,d)$, one is naturally inclined to form its completion $(\overline{X},\overline{d})$.  Here, by the completion of a pseudometric space, we mean the completion of its separation.  Concretely, we view $\overline{X}$ as the space of equivalence classes of Cauchy sequences $(x_n)$ from $X$ modulo the pseudometric $d((x_n),(y_n)):=\lim_n d(x_n,y_n)$; this pseudometric naturally induces a metric $\overline{d}$ on $\overline{X}$.

For each of notation, we set $\mathcal{F}_{\vec x}:=\overline{\mathcal{F}_{\vec x}^0}$.

\begin{df}
A \textbf{$T$-formula}\footnote{In the literature, this is called a \textbf{definable predicate in $T$} but we find the term $T$-formula much more suggestive.} (over $\vec x$) is simply an element of $\mathcal{F}_{\vec x}$.
\end{df}

\begin{nrmk}
With this terminology, a formula from $\F_{\vec x}^0$ is not technically a $T$-formula, but rather its $T$-equivalence class is a $T$-formula.  Since $T$-equivalent formulae might as well be treated as equal, this abuse of terminology is not troublesome.  An element of $\F_{\vec x}^0$ (viewed as a $T$-formula) will be referred to as a \textbf{simple} $T$-formula.
\end{nrmk}

Consider $\varphi\in \mathcal{F}_{\vec x}$ and take a Cauchy sequence $(\varphi_n(\vec x))$ from $\F_{\vec x}^0$ that represents $\varphi$.  Then, for every $\epsilon>0$, there is $N$ such that, for all $m\geq n\geq N$, we have $d(\varphi_m,\varphi_n)\leq \epsilon$, or, in other words:
$$T\models \sup_{\vec x}|\varphi_m(\vec x)-\varphi_n(\vec x)|\leq \epsilon.$$  Note also that if $(\psi_n(x))$ is another Cauchy sequence from $\F_{\vec x}^0$ representing $\varphi$, then, there is $N'$ such that, for all $m\geq N'$, we have $d(\varphi_m,\psi_m)\leq \epsilon$, or, in other words:
$$T\models \sup_{\vec x}|\varphi_m(\vec x)-\psi_m(\vec x)|\leq \epsilon.$$
Consequently, in all $\mathbf{M}\models T$, we have a well-defined \textbf{interpretation} $\varphi^{\mathbf{M}}$ of $\varphi$ given by $\varphi^{\mathbf{M}}(\vec a):=\lim_n \varphi_n^\mathbf{M}(\vec a)$.\footnote{The fact that we have such an interpretation yields credence to the use of the term $T$-formula.}  Note also that $\varphi^\textbf{M}$ is a uniformly continuous function which takes values in a compact interval $K_\varphi$ in $\r$ in a way that does not depend on $\mathbf{M}$.


\begin{ex}
Suppose that $(\psi_m(\vec x))$ is any sequence from $\F_{\vec x}^0$.  For $n\in \n$, set $\varphi_n(\vec x):=\sum_{m\leq n}2^{-m}\psi_m(\vec x)$.  It is clear that $(\varphi_n)$ is a Cauchy sequence from $\F_{\vec x}^0$, whence represents a $T$-formula.  For psychological reasons, we denote this $T$-formula by $\sum_m 2^{-m}\psi_m$.
\end{ex}

In connection with the main result of the next section, it will prove useful to end this subsection with a brief discussion on type-spaces.

\begin{df}

\

\begin{enumerate}
\item Given $\mathbf{M}\models T$ and $\vec a\in \mathbf{M}^{\vec x}$, we define the \textbf{type of $\vec a$ in $\mathbf{M}$} to be the function $\tp^{\mathbf{M}}(\vec a):\mathcal{F}_{\vec x}\to \r$ given by $\tp^{\mathbf{M}}(\vec a)(\varphi):=\varphi^{\mathbf{M}}(\vec a)$.  
\item A function $p:\mathcal{F}_{\vec x}\to \r$ is called a \textbf{type in $\vec x$ over $T$} if $p=\tp^{\mathbf{M}}(\vec a)$ for some $\mathbf{M}\models T$ and some $\vec a\in \mathbf{M}^{\vec x}$; in this case, we say that $\vec a$ \textbf{realizes} $p$.
\item The set of types in $T$ over $\vec x$ will be denoted by $S_{\vec x}(T)$.
\item Given $\varphi\in \mathcal{F}_{\vec x}$, we define $f_\varphi:S_{\vec x}(T)\to \r$ by $f_\varphi(p):=p(\varphi)$.  The \textbf{logic topology} on $S_{\vec x}(T)$ is the weakest topology making all maps $f_{\varphi}$ continuous.  The Compactness Theorem yields that $S_{\vec x}(T)$ is compact with respect to the logic topology.\footnote{A remark for the analysts:  $\overline{\mathcal{F}_{\vec x}}$ is naturally a Banach over $\r$.  Note then that a type is a continuous functional on $\overline{\mathcal{F}_{\vec x}}$ and an easy ultraproduct argument shows that $S_{\vec x}(T)$ is a closed subset of $(\overline{\mathcal{F}_{\vec x}})^*$ in the weak*-topology; the logic topology on $S_{\vec x}(T)$ is merely the restriction of the weak*-topology.  The compactness of $S_{\vec x}(T)$ follows from this observation.}
\end{enumerate}
\end{df}

It is straightforward to verify that the set of functions $f_\varphi$ separate points in $S_{\vec x}(T)$, the Stone-Weierstrass theorem yields:

\begin{thm}\label{SW}
A function $f:S_{\vec x}(T)\to \r$ is continuous if and only if there is a $T$-formula $\varphi(\vec x)$ for which $f=f_\varphi$.
\end{thm}

%

\subsection{Definability relative to a theory}

In order to smoothly state the main theorem of this section, we need to collect a few elementary (if not slightly cumbersome) definitions and examples.  

In the following definition, we view $\Mod(T)$ as a category whose objects are models of $T$ and whose morphism are elementary embeddings.  We also let $\Met$ be the category whose objects are bounded metric spaces and whose morphisms are isometric embeddings.

\begin{df}
A functor $X:\Mod(T)\to \Met$ is called a \textbf{$T$-functor over $\vec x$} if:
\begin{itemize}
\item For every $\mathbf{M}\models T$, $X(\mathbf{M})$ is a closed subset of $\mathbf{M}^{\vec x}$, and 
\item $X$ is restriction on morphisms.
\end{itemize}
\end{df}

A natural source of $T$-functors:

\begin{df}
Given a $T$-formula $\varphi(\vec x)$, its \textbf{zeroset} is the $T$-functor $Z(\varphi)$ given by $$Z(\varphi)(\mathbf{M}):=Z(\varphi^{\mathbf{M}}):=\{\vec a \in \mathbf{M}^{\vec x} \ : \ \varphi^{\mathbf{M}}(\vec a)=0\}.$$
\end{df}

We will also need the following notion:

\begin{df}
A \textbf{$T$-function over $\vec x$} is simply a function whose domain is the set of pairs $(\mathbf{M},\vec a)$, with $\mathbf{M}\models T$ and $\vec a\in \mathbf{M}^{\vec x}$, and whose codomain is a bounded subset of $\r$.  A $T$-function is \textbf{nonnegative} if its co-domain is contained in the set of nonnegative reals.
\end{df}

A natural source of $T$-functions:

\begin{ex}
The interpretation of a $T$-formula is a $T$-function.  If the $T$-function $\Phi$ is the interpretation of a $T$-formula $\varphi$, we will say that $\Phi$ is \textbf{realized} by $\varphi$.  If $\Phi$ is nonnegative, then we also say that $\varphi$ is nonnegative.
\end{ex}

\begin{ex}\label{quantifyfunctor}
Suppose that $X$ is a $T$-functor over $\vec x$.
\begin{itemize}
\item For any $T$-function $\Phi$ over $(\vec x,\vec y)$, we have the $T$-functions $$\sup_{\vec x}\Phi(\vec x,\vec y) \quad \text{ and }\quad \inf_{\vec x}\Phi(\vec x,\vec y)$$ over $\vec y$ defined by mapping the pair $(\mathbf{M},\vec b)$ to $$\sup_{\vec a\in X(\mathbf{M})}\Phi(\mathbf{M},\vec a,\vec b)\quad \text{ and }\quad\inf_{\vec a\in X(\mathbf{M})}\Phi(\mathbf{M},\vec a,\vec b)$$ respectively.
\item A particular instance of the previous item is the case that $\Phi$ is simply the interpretation of the $T$-formula $d(\vec x,\vec y)$ (where $\vec x$ and $\vec y$ range over the same product of sorts).  In this case, we see that the $T$-function which maps $(\mathbf{M},\vec b)$ to $d(\vec b,X(\mathbf{M}))$ is a nonnegative $T$-function, which we write suggestively as $d(\vec x,X)$.
\end{itemize}
\end{ex}

\begin{df}
We say that a nonnegative $T$-function $\Phi$ over $\vec x$  is \textbf{almost-near} if, for every $\epsilon>0$, there is $\delta>0$ such that, for all $(\mathbf{M},\vec a)$, if $\Phi(\mathbf{M},\vec a)<\delta$, then there is $\vec b\in \mathbf{M}^{\vec x}$ such that $d(\vec a,\vec b)\leq \epsilon$ and with $\Phi(\mathbf{M},\vec b)=0$.  We say that a $T$-formula is almost-near if its interpretation is almost-near.\footnote{In \cite{omitting}, almost-near formulae are called \emph{stable} and in \cite{munster} they are called \emph{weakly stable}.  While these terminologies arise from corresponding terminology in the operator algebra literature, they are infinitely confusing to model-theorists (as they have nothing to do with the model-theoretic notion of stable formula) and thus we prefer the more suggestive term almost-near.}
\end{df}

\begin{ex}\label{distancean}
Given a $T$-functor $X$ over $\vec x$, the $T$-function $d(\vec x,X)$ is an almost-near $T$-function (simply take $\delta=\epsilon$).
\end{ex}

With all of the terminology established, we can now neatly state the main theorem of this section:

\begin{thm}\label{maindefinabilitytheorem}
Suppose that $X$ is a $T$-functor over $\vec x$.  Then the following are equivalent:
\begin{enumerate}
\item For all $T$-formulae $\psi(\vec x, \vec y)$, the $T$-functions $\sup_{\vec x\in X}\psi(\vec x,\vec y)$ and $\inf_{\vec x\in X}\psi(\vec x,\vec y)$ are realized by $T$-formulae.
\item The $T$-function $d(\vec x,X)$ is realized by a $T$-formula.
\item There is an almost-near $T$-formula $\varphi(\vec x)$ such that $X=Z(\varphi)$.
\item For all sets $I$, all families of models $(\mathbf{M}_i)_{i\in I}$ of $T$, and all ultrafilters $\u$ on $I$, we have
$$X(\prod_\u \mathbf{M}_i)=\prod_\u X(\mathbf{M}_i).$$
\item
\begin{enumerate}
\item $X$ is a zeroset, and
\item for any nonnegative $T$-formula $\varphi(\vec x)$ such that $X=Z(\varphi)$, we have that $\varphi$ is almost-near.
\end{enumerate}
\end{enumerate}
\end{thm}

\begin{proof}
(1) implies (2) is immediate and (2) implies (3) follows from Example \ref{distancean}.

(3) implies (4):  Suppose that $X=Z(\varphi)$ for an almost-near $T$-formula $\varphi(\vec x)$ and fix a family $(\mathbf{M}_i)_{i\in I}$ of models of $T$ and an ultrafilter $\u$ on $I$.  Set $\mathbf{M}:=\prod_\u \mathbf{M}_i$.  Notice that $\prod_\u X(\mathbf{M}_i)\subseteq X(\mathbf{M})$ always holds by the \L os theorem, and, in fact, $\prod_\u X(\mathbf{M}_i)$ is a closed subset of $X(\mathbf{M})$.  Now consider $\vec a=(\vec a_i)^\bullet\in X(\mathbf{M})$ and fix $\epsilon>0$.  Take $\delta>0$ witnessing that $\varphi$ is an almost-near formula for $\epsilon$.  Since $\varphi^{\mathbf{M}}(\vec a)=0$, there is $I\in \u$ such that $\varphi^{\mathbf{M}}(\vec a_i)<\delta$ for $i\in I$.  By choice of $\delta$, for $i\in I$ we may find $\vec b_i\in X(\mathbf{M}_i)$ such that $d(\vec a_i,\vec b_i)\leq \frac{1}{n}$.  Set $\vec b:=(\vec b_i)^\bullet \in \prod_\u X(\mathbf{M}_i)$ (with $\vec b_i\in \mathbf{M}_i$ chosen arbitrarily for $i\notin I$).  It follows that $d(\vec a,\vec b)\leq \epsilon$.  Since $\prod_\u X(\mathbf{M}_i)$ is closed, it follows that $\vec a\in \prod_\u X(\mathbf{M}_i)$.  

(4) implies (5):  We assume that (4) holds and first prove (a).  Note that it suffices to show that the $T$-function $d(\vec x,X)$ is realized by a $T$-formula.  Towards this end, by Theorem \ref{SW}, it suffices to show that there is a continuous function $f:S_{\vec x}(T)\to \r$ given by $f(p):=d(\vec a,X(\mathbf{M}))$ for any $\mathbf{M}\models T$ and $\vec a\in \mathbf{M}^{\vec x}$ realizing $p$.  

In order to show that the function $f$ from the previous paragraph is well-defined, suppose, towards a contradiction, that there are $\mathbf{M},\mathbf{N}\models T$, $\vec a\in \mathbf{M}^{\vec x}$ and $\vec b\in \mathbf{N}^{\vec x}$ such that $\vec a$ and $\vec b$ both realize $p$ yet $d(\vec a,X(\mathbf{M}))<d(\vec b,X(\mathbf{N}))$.  Fix a nonprincipal ultrafilter $\u$ on $\n$ and take an elementary embedding $i:\mathbf{M}\to \mathbf{N}^{\u}$ with $i(\vec a)=\vec b$.  Take $\vec c\in X(\mathbf{M})$ such that $d(\vec a,\vec c)<d(\vec b,X(\mathbf{N}))$.  We then arrive at the contradiction
$$d(\vec a,\vec c)=d(\vec b,i(\vec c))\geq d(\vec b,X(\mathbf{N}^\u))=d(\vec b,X(\mathbf{N})^\u)\geq d(\vec b,X(\mathbf{N})).$$

It remains to prove that $f$ is continuous.  We thus suppose that $(p_i)_{i\in I}$ is a net from $S_{\vec x}(T)$ and $\u$ is an ultrafilter on $I$ such that $\lim_\u p_i=p$; we must show that $\lim_\u f(p_i)=f(p)$.  For each $i\in I$, take $\mathbf{M}_i\models T$ and $\vec a_i\in \mathbf{M}_i$ realizing $p_i$.  Set $\mathbf{M}:=\prod_\u \mathbf{M}_i$ and set $\vec a:=(\vec a_i)^\bullet$.  Notice that $\vec a$ realizes $p$.  We thus have
$$f(p)=d(\vec a,X(\mathbf{M}))=d(\vec a,\prod_\u X(\mathbf{M}_i))=\lim_\u d(\vec a_i,X(\mathbf{M}_i))=\lim_\u f(p_i).$$
This finishes the proof of (a).


We now prove (b).  Suppose that $\varphi$ is nonnegative, $X=Z(\varphi)$, and yet $\varphi$ is not almost-near.  There is thus some $\epsilon>0$ such that, for all $n\in \n$, there are $\mathbf{M}_n\models T$ and $\vec a_n\in \mathbf{M}_n$ such that $\varphi^{\mathbf{M}_n}(\vec a_n)<\frac{1}{n}$ and yet $d(\vec a_n,X(\mathbf{M}_n))\geq \epsilon$.  It follows that $\vec a\in X(\prod_\u \mathbf{M}_n)$, whence, by (4), we have that, for $\u$-almost all $n$, there are $\vec b_n\in X(\mathbf{M}_n)$ such that $d(\vec a_n,\vec b_n)<\epsilon$, a contradiction.

(5) implies (1):  We only prove the $\inf$ case, the $\sup$ case being similar.  We first need an elementary fact from analysis \cite[Lemma 2.10/Remark 2.12]{BBHU}:  given any function $\Delta:(0,1]\to (0,1]$, there is an increasing, continuous function $\alpha:[0,1]\to [0,1]$ such that:  for any set $Y$ and functions $F,G:Y\to [0,1]$ satisfying
$$(\forall \epsilon>0)(\forall x\in Y)(F(x)\leq \Delta(\epsilon)\Rightarrow G(x)\leq \epsilon), \quad (\dagger)$$  we have $G(x)\leq \alpha(F(x))$ for all $x\in Y$.  

Next fix a nonnegative $T$-formula $\varphi(\vec x)$ such that $X=Z(\varphi)$.  By (5), $\varphi$ is an almost-near formula, whence there is a function $\Delta:(0,1]\to (0,1]$ for which $(\dagger)$ holds whenever $Y:=\mathbf{M}^{\vec x}$, $F(\vec x):=\varphi^{\mathbf{M}}(\vec x)$, and $G(\vec x):=d(\vec x,X(\mathbf{M}))$ (for $\mathbf{M}\models T$).  Consequently, there is an $\alpha$ as in the previous paragraph such that for all $\mathbf{M}\models T$ and all $\vec a\in \mathbf{M}^{\vec x}$, we have $d(\vec a,X(\mathbf{M}))\leq \alpha(\varphi^{\mathbf{M}}(\vec a))$.

Now fix an arbitrary $T$-formula $\psi(\vec x,\vec y)$.  By quoting the first paragraph again, there is a function $\beta$ such that, for all $\mathbf{M}\models T$ and all $\vec a\in \mathbf{M}^{\vec x}$ and $\vec b,\vec c\in \mathbf{M}^{\vec y}$, we have $|\psi^{\mathbf{M}}(\vec a,\vec b)-\psi^{\mathbf{M}}(\vec a,\vec c)|\leq \beta(d(\vec b,\vec c))$.  It is now straightforward to check that $\inf_{y\in X}\psi(\vec x,\vec y)$ is realized by the $T$-formula $\inf_{\vec z}[\psi(\vec x,\vec z)+\beta(\alpha(\varphi(z)))]$.
\end{proof}


\begin{df}
A $T$-functor satisfying any of the equivalent conditions in the previous theorem is called a \textbf{$T$-definable set}.
\end{df}

\begin{nrmk}
In the literature, item (2) in the previous definition is often given as the definition of definable set.  It is our opinion that this choice of definition may seem a bit obscure at first.  However, we hope that it is clear that items (1) and (4) are desirable properties of a $T$-functor. 
\end{nrmk}

%
%


\subsection{Definability in a structure}
For our purposes, it will also be convenient to have a notion of a definable subset of a particular structure.  For the rest of this section, we abandon our fixed theory $T$ from above and instead fix a structure $\mathbf{M}$ and a subset $A\subseteq \mathbf{M}$.  Suppose that $(\varphi_n(\vec x))$ is a sequence of $L(A)$-formulae that are uniformly Cauchy in $\mathbf{M}$, that is, for all $\epsilon>0$, there is $N\in \n$ such that for all $m,n\geq N$ and all $\vec a\in \mathbf{M}^{\vec x}$, we have that $|\varphi_m^{\mathbf{M}}(\vec a)-\varphi_n^{\mathbf{M}}(\vec a)|\leq \epsilon$.  Then the same fact remains true in any model of $\Th(\m_A)$, whence we have that $\varphi:=(\varphi_n(x))$ is a $\Th(\mathbf{M}_A)$-formula.  We will also say that $\varphi$ is a \textbf{formula in $\mathbf{M}$ over $A$}.  It is clear that in any model $\mathbf{N}$ of $\Th(\mathbf{M}_A)$, one has a natural interpretation $\varphi^{\mathbf{N}}$ of $\varphi$ in $\mathbf{N}$.

\begin{df}
A closed subset $X\subseteq \mathbf{M}^{\vec x}$ is called \textbf{definable in $\mathbf{M}$ over $A$} if the function $\vec x\mapsto d(\vec x,X):\mathbf{M}^{\vec x}\to \r$ is the interpretation of a formula in $\mathbf{M}$ over $A$.
\end{df}

Some of the equivalences in Theorem \ref{maindefinabilitytheorem} remain true in this local context:

\begin{thm}\label{localdefinabilitytheorem}
Suppose that $X\subseteq \mathbf{M}^{\vec x}$ is a closed subset and $A\subseteq \mathbf{M}$.  The following are equivalent:
\begin{enumerate}
\item $X$ is definable in $\mathbf{M}$ over $A$.
\item For all formulae $\psi(\vec x, \vec y)$ in $\mathbf{M}$ over $A$, there are formulae in $\mathbf{M}$ over $A$ whose interpretations in $\mathbf{M}$ coincide with the functions
$$\sup_{\vec x\in X}\psi^{\mathbf{M}}(\vec x,\vec y) \quad \text{ and } \quad \inf_{\vec x \in X}\psi^{\mathbf{M}}(\vec x,\vec y).$$
\item There is a formula $\varphi(\vec x)$ in $\mathbf{M}$ over $A$ with $X=Z(\varphi^{\mathbf{M}})$ and such that: for all $\epsilon>0$, there is a $\delta>0$ such that, for all $\vec a\in \mathbf{M}^{\vec x}$, 
$$\varphi^{\mathbf{M}}(\vec a)<\delta\Rightarrow d(\vec a,X)\leq \epsilon.$$
\end{enumerate}
\end{thm}

We leave the proof of the previous theorem to the reader as it follows many of the ideas in the proof of Theorem \ref{maindefinabilitytheorem}.  A key distinction in the local theory of definability is that the local analog of item (5) in Theorem \ref{maindefinabilitytheorem} is no longer true.  First, a definition:   

\begin{df}
Suppose that $\varphi$ is a formula in $\mathbf{M}$ over $A$.  We will say that $Z(\varphi^{\mathbf{M}})$ is \textbf{$\varphi$-definable} if for every $\epsilon>0$, there is $\delta>0$ such that, for all $a\in \mathbf{M}^{\vec x}$, if $\varphi(\vec a)^{\mathbf{M}}<\delta$, then $d(\vec a,Z(\varphi^{\mathbf{M}}))\leq \epsilon$.
\end{df}

Item (3) of Theorem \ref{localdefinabilitytheorem} can be recast in this new terminology:

\begin{cor}
Suppose that $X\subseteq \mathbf{M}^{\vec x}$ is a closed subset and $A\subseteq \mathbf{M}$.  Then $X$ is definable if and only if it is $\varphi$-definable for some formula $\varphi(\vec x)$ in $\mathbf{M}$ over $A$.
\end{cor}

Unfortunately, it is not true that if $\varphi$ is a formula in $\mathbf{M}$ over $A$ and $Z(\varphi^{\mathbf{M}})$ is definable, then it is $\varphi$-definable.\footnote{It is a good exercise for the reader to see where the proof of the corresponding part of Theorem \ref{maindefinabilitytheorem} breaks down in the local situation.}  In fact, we will see concrete instances of this distinction in Remarks \ref{unfortunate1} and \ref{unfortunate2} below.  However, the following  characterization of $\varphi$-definability shows that it is a natural notion:

\begin{thm}\label{ultrapowerlocaldefinability}
Suppose that $\varphi$ is a formula in $\mathbf{M}$ over $A$.  Then $Z(\varphi)$ is $\varphi$-definable if and only if $Z(\varphi^{\mathbf{M}})^\u=Z(\varphi^{\mathbf{M}^\u})$ for every ultrafilter $\u$.
\end{thm}

%

We leave the proof of the previous theorem to the reader as it is extremely similar to the earlier proofs in this section.  We end this section by observing that, in \emph{saturated} structures, there is no distinction between definable and $\varphi$-definable:

\begin{prop}\label{saturateddefinable}
Suppose that $\mathbf{M}$ is $\aleph_1$-saturated and $\varphi$ is a formula in $\mathbf{M}$ over $A$.  Then $Z(\varphi)$ is definable if and only if it is $\varphi$-definable.
\end{prop}

Once again, given the earlier proofs in this section, the proof of the previous proposition is rather routine.  One can also consult \cite[Remark 9.20]{BBHU}.


\section{Spectral gap for unitary group representations}

\subsection{Generalities on unitary group representations}  We begin by fixing some terminology concerning Hilbert spaces.  Throughout, we suppose that $\H$ is a complex Hilbert space.  We let $\B(\H)$ denote the set of bounded operators on $\H$, that is, the set of linear operators $T:\H\to \H$ for which the quantity
$$\|T\|:=\sup\{\|Tx\| \ : \ \|x\|\leq 1\}$$ is finite.  We refer to $\|T\|$ as the \textbf{operator norm} of $T$.  Recall that for $T\in \B(\H)$, the \textbf{adjoint} of $T$ is the unique operator $T^*:\H\to \H$ satisfying $\langle Tx,y\rangle=\langle x,T^*y\rangle$ for all $x,y\in \H$.  It is straightforward to check that $T^*$ is also an element of $\B(\H)$.  $\B(\H)$, equipped with addition, composition, scalar multiplication, and adjoint has the structure of a \textbf{$*$-algebra}.  Finally, we let $\U(H)$ denote the subset of $\B(\H)$ consisting of \textbf{unitary operators} on $\H$, that is, those $T\in \B(\H)$ for which $T^*=T^{-1}$.  Note that $\U(\H)$ is a subgroup of $\B(\H)$ under composition.

\emph{Throughout this section, we fix a \emph{countable} group $\Gamma$}.  A \textbf{unitary representation of $\Gamma$} is simply a group homomorphism $\pi:\Gamma\to \U(\H_{\pi})$ for some Hilbert space $\H_{\pi}$.  

\begin{ex}
Given $\Gamma$, we set $$\ell^2\Gamma:=\{f:\Gamma \to \mathbb{C} \ : \ \sum_{\gamma\in \Gamma}|f(\gamma)|^2<\infty\}.$$  $\ell^2\Gamma$ is naturally a Hilbert space under the inner product $\langle f,g\rangle:=\sum_{\gamma\in \Gamma} f(\gamma)\overline{g(\gamma)}$.  
The \textbf{left-regular representation} of $\Gamma$ is the unitary representation $\lambda_\Gamma:\Gamma\to \U(\ell^2\Gamma)$ given by $(\lambda_\Gamma(\gamma)(f))(\eta):=f(\gamma^{-1}\eta)$.
\end{ex}

Suppose that $\pi$ is a unitary representation of $\Gamma$.  A closed subspace $\mathcal{K}\subseteq \H_\pi$ is called \textbf{invariant under $\pi$} if $\pi(\gamma)(\mathcal{K})\subseteq \mathcal{K}$ for all $\gamma\in \Gamma$.  In this case, we may consider the \textbf{restriction} of $\pi$ to $\mathcal{K}$, denoted $\pi|\mathcal{K}:\Gamma\to \U(\mathcal{K})$.

Recall that, given a subspace $\mathcal{K}$ of $\H$, the \textbf{orthogonal complement} of $\mathcal{K}$ is the set $$\mathcal{K}^{\perp}:=\{\zeta \in \H \ : \ \langle \zeta,\eta\rangle=0 \text{ for all }\eta\in \mathcal{K}\}.$$  Note that $\mathcal{K}^{\perp}$ is a closed subspace of $\H$, whence a Hilbert space in its own right.  If $\H=\H_{\pi}$ for some unitary representation $\pi$ of $\Gamma$ and $\mathcal{K}$ happens to be invariant under $\pi$, then it is readily verified that $\mathcal{K}^{\perp}$ is also invariant under $\pi$.

Given a unitary representation $\pi$ of $\Gamma$, we set $$\Fix(\pi):=\{\zeta \in \H_{\pi} \ : \ \pi(\gamma)(\zeta)=\zeta \text{ for all }\gamma\in \Gamma\}.$$  Note that $\Fix(\pi)$ is an invariant subspace of $\H_{\pi}$.  We say that $\pi$ is \textbf{ergodic} if $\Fix(\pi)=\{0\}$.  We set $\Erg(\pi):=\Fix(\pi)^{\perp}$ and refer to it as the \textbf{ergodic part} of $\pi$.  Of course, $\pi|\Erg(\pi)$ is ergodic.

\begin{ex}\label{finitefix}
$\lambda_\Gamma$ is ergodic if and only if $\Gamma$ is infinite.
\end{ex}

Finally, given a family $(\pi_i)_{i\in I}$ of unitary representations of $\Gamma$ and an ultrafilter $\u$ on $I$, one can consider the unitary representation $\prod_\u \pi_i:\Gamma\to \U(\prod_\u \H_i)$ given by $(\prod_\u \pi_i(\gamma))(\xi_i)^\bullet:=(\pi_i(\gamma)(\xi_i))^\bullet$.  We refer to $\prod_\u \pi_i$ as the \textbf{ultraproduct} of the representations $\pi_i$.  If $\pi_i=\pi$ for all $i$, we write $\pi^\u$ for the \textbf{ultrapower} of $\pi$.

\subsection{Introducing spectral gap}  We leave the proof of the following proposition to the reader:

\begin{prop}\label{spectgap}
Let $\pi:\Gamma\to \U(\H_{\pi})$ be a unitary representation.  The following are equivalent:
\begin{enumerate}
\item There exists finite $F\subseteq \Gamma$ and $c>0$ such that, for all $\zeta\in \H_{\pi}$, we have
$$\max_{\gamma\in F}\|\pi(\gamma)\zeta-\zeta\|\geq c\|\zeta\|.$$
\item For any nonprincipal ultrafilter $\u$, $\pi^\u$ is ergodic.
\item For all $\epsilon>0$, there is a finite $F\subseteq \Gamma$ and $\delta>0$ such that, for all $\zeta\in H_{\pi}$, we have
$$\max_{\gamma\in F}\|\pi(\gamma)\zeta-\zeta\|\leq \delta \Rightarrow \|\zeta\|\leq \epsilon.$$
\end{enumerate}
\end{prop}

Notice that an action satisfying the equivalent properties enumerated in Proposition \ref{spectgap} is necessarily ergodic.  Here is arguably the most important definition in this entire paper.

\begin{df}
A unitary representation $\pi$ has \textbf{spectral gap} if $\pi|\Erg(\pi)$ satisfies the equivalent properties enumerated in Proposition \ref{spectgap}.
\end{df}

\begin{ex}\label{finitegap}
If $\Gamma$ is finite, then every representation of $\Gamma$ has spectral gap.  To see this, fix an ergodic representation $\pi$ of $\Gamma$ and $\xi\in \H_\pi$.  Note then that $\frac{1}{|\Gamma|}\sum_{\gamma\in \Gamma}\pi(\gamma)\xi$ belongs to $\Fix(\pi)$ and is thus equal to $0$.  It follows that
$$\|\xi\|=\left\|\frac{1}{|\Gamma|}\sum_{\gamma\in \Gamma}\pi(\gamma)\xi-\xi\right\|=\left\|\frac{1}{|\Gamma|}\sum_{\gamma\in \Gamma}(\pi(\gamma)\xi-\xi)\right\|\leq \frac{1}{|\Gamma|}\sum_{\gamma\in \Gamma}\|\pi(\gamma)\xi-\xi\|.$$  Thus, if $\max_{\gamma\in \Gamma}\|\pi(\gamma)\xi-\xi\|\leq \epsilon$ for all $\gamma\in \Gamma$, we have that $\|\xi\|\leq \epsilon$.
\end{ex}

\begin{ex}\label{amenablespec}
By a theorem of Hulanicki and Reiter (see \cite[Theorem G.3.2]{Tbook}), if $\Gamma$ is infinite, then $\lambda_\Gamma$ has spectral gap if and only if $\Gamma$ is non-amenable.
\end{ex}

Before we move on any further, let us briefly explain the terminology.  Suppose that $\pi$ is an ergodic unitary representation of $\Gamma$.  Suppose that $F\subseteq \Gamma$ is finite and closed under inverse.  Define $h_F:=\frac{1}{|F|}\sum_{\gamma\in F}\pi(\gamma)$, a self-adjoint contraction, that is, $h_F^*=h_F$ and $\|h_F\|\leq 1$.  Consequently, $\sigma(h_F)\subseteq [-1,1]$, where $\sigma(h_F)$ denotes the \textbf{spectrum} of $h_F$, namely $$\sigma(h_F):=\{\lambda \in \mathbb{C} \ : \ h_F-\lambda\cdot I \text{ is not invertible}\}.$$

The following fact explains the terminology spectral gap; see, for example, \cite[Corollary 15.1.4]{AP}.

\begin{fact}
$\pi$ has spectral gap if and only if there is a symmetric finite $F$ and $\delta<1$ such that $\sigma(h_F)\subseteq [-1,1-\delta]$.
\end{fact}

In the next section, the following theorem will immediately yield the connection between spectral gap and definability.

\begin{thm}\label{ultrapowerspecgap}
A unitary representation $\pi$ has spectral gap if and only if, for any nonprincipal ultrafilter $\u$, we have $\Fix(\pi^\u)=\Fix(\pi)^\u$.
\end{thm}

\begin{proof}
Note that $\Fix(\pi^\u)=\Fix(\pi)^\u\oplus \Fix((\pi|\Erg(\pi))^\u)$.  By part (2) of Proposition \ref{spectgap}, $\pi$ has spectral gap if and only if $\Fix((\pi|\Erg(\pi))^\u)=\{0\}$, whence the result follows.
\end{proof}

\subsection{Spectral gap and definability}  For the rest of this section, we fix an enumeration of $\Gamma$, say $\Gamma=\{\gamma_0,\gamma_1,\gamma_2,\ldots\}$.  We view a unitary representation $\pi$ as a structure in the language of Hilbert spaces\footnote{Say, for simplicity, the one-sorted language for the unit ball of Hilbert spaces.} augmented by function symbols for elements of $\Gamma$ in the natural way.  It follows easily that the class of unitary representations of $\Gamma$ is in fact an axiomatizable class in this language, say $\Mod(T_{\Gamma})$.  Note that $\pi\mapsto \Fix(\pi)$ is a $T_\Gamma$-functor and, in fact, $\Fix(\pi)$ is the zeroset of the $T_\Gamma$-formula $$\varphi_{\Gamma}(x):=\sum_m 2^{-m}\|\gamma_m\cdot x-x\|.$$  

By Theorems \ref{ultrapowerlocaldefinability} and \ref{ultrapowerspecgap}, we immediately have:

\begin{cor}\label{specgapdef}
Fix a unitary representation $\pi$ of $\Gamma$.  Then $\pi$ has spectral gap if and only if $\Fix(\pi)$ is a $\varphi_\Gamma$-definable subset of $\H_{\pi}$.
\end{cor}

%

\begin{nrmk}\label{unfortunate1}
Unfortunately, in general, one cannot replace ``$\varphi_\Gamma$-definable'' in the previous theorem with ``definable.''  For example, suppose that $\Gamma$ is an infinite amenable group.  Then by Examples \ref{finitefix} and \ref{amenablespec}, $\Fix(\lambda_\Gamma)=\{0\}$ (which is clearly a definable subset of $\ell^2\Gamma$) but $\lambda_\Gamma$ does not have spectral gap.  However, as shown in Proposition \ref{saturateddefinable}, if $\H_\pi$ is $\aleph_1$-saturated, then it is in fact true that $\pi$ has spectral gap if and only if $\Fix(\pi)$ is definable.    
\end{nrmk}

We now address the global question.  First, we need a definition:

\begin{df}
We say that $\Gamma$ has \textbf{property (T)} if every unitary representation of $\Gamma$ has spectral gap.
\end{df}

\begin{ex}

\

\begin{enumerate}
\item Finite groups have property (T).  This follows from the calculation done in Example \ref{finitegap}.
\item If $n\geq 3$, then $\operatorname{SL}_n(\mathbb{Z})$ has property (T).  This is a theorem due to Kazhdan; see \cite[Section 1.4]{Tbook}.
\item It has recently been shown in \cite{KNO} (using a computer-assisted proof) that $\Aut(\mathbb{F}_5)$ has property (T).
\item \emph{Random groups} (in the sense of Gromov) have property (T); see the introduction to \cite{Tbook} for references.
\item By Examples \ref{finitefix} and \ref{amenablespec}, infinite amenable groups never have property (T).  A finitely generated free group also does not have property (T); see \cite[Example 1.3.7]{Tbook}.
\end{enumerate}
\end{ex}

While the above definition of property (T) is probably not the standard one, we find it be the most natural given the context of this article.  We now describe the usual definition.  First, given a unitary representation $\pi$ of $\Gamma$, a finite subset $F$ of $\Gamma$, and $\delta>0$, we say that $\xi\in \H_\pi$ is \textbf{$(F,\delta)$-almost invariant} if $\max_{\gamma \in F}\|\pi(\gamma)(\xi)-\xi\|<\delta \|\xi\|$.  We say that $\pi$ has \textbf{almost invariant vectors} if, for every finite subset $F$ of $\Gamma$ and every $\delta>0$, $\pi$ has a $(F,\delta)$-almost invariant vector.  Note that $\pi$ has almost invariant vectors if and only if there is a nonprincipal ultrafilter $\u$ such that $\pi^\u$ is not ergodic.  The following lemma is now immediate:

\begin{lemma}
$\Gamma$ has property (T) if and only if:  for every unitary representation $\pi$ of $\Gamma$, if $\pi$ has almost invariant vectors, then $\pi$ is not ergodic.
\end{lemma}

It turns out that one can improve the definition of property (T) using a notion that is a priori weaker:

\begin{lemma}\label{kazhdan}
$\Gamma$ has property (T) if and only if there is a finite $F\subseteq \Gamma$ and $\delta>0$ such that:  for every unitary representation $\pi$, if $\pi$ has a $(F,\delta)$-almost invariant vector, then $\pi$ is not ergodic.
\end{lemma}

\begin{proof}[Proof Sketch]
Suppose that no such pair $(F,\delta)$ exists.  Then, for each such pair, there is an ergodic  representation $\pi_{(F,\delta)}$ of $\Gamma$ with an $(F,\delta)$-almost invariant vector.  It is easy to verify that $\bigoplus_{(F,\delta)}\pi_{(F,\delta)}$ is an ergodic representation of $\Gamma$ with almost invariant vectors, whence $\Gamma$ does not have property (T).
\end{proof}

A pair $(F,\delta)$ as in Lemma \ref{kazhdan} is called a \textbf{Kazhdan pair} for $\Gamma$ and $F$ is called a \textbf{Kazhdan set} for $\Gamma$.  The following proposition gets us closer to the connection with definability:

\begin{prop}\label{defT}
Suppose that $(F,\delta)$ is a Kazhdan pair for $\Gamma$.  Then for any unitary representation $\pi$ of $\Gamma$ and any $\epsilon>0$, if $\xi\in \H_{\pi}$ is $(F,\delta\epsilon)$-invariant, then there is $\eta\in \Fix(\pi)$ such that $\|\xi-\eta\|<\epsilon\|\xi\|$.
\end{prop}

\begin{proof}
Write $\xi=\xi_1+\xi_2$ with $\xi_1\in \Fix(\pi)$ and $\xi_2\in \Erg(\pi)$; it suffices to show that $\|\xi_2\|<\epsilon\|\xi\|$.  By the definition of Kazhdan pair, there is $\gamma\in F$ such that $\|\pi(\gamma)(\xi_2)-\xi_2\|\geq \delta\|\xi_2\|$.  On the other hand, we have 
$$\|\pi(\gamma)(\xi_2)-\xi_2\|=\|\pi(\gamma)(\xi)-\xi\|<\delta\epsilon \|\xi\|.$$  The desired result now follows.
\end{proof}




\begin{thm}
The following are equivalent:
\begin{enumerate}
\item $\Gamma$ has property (T).
\item The $T_\Gamma$-functor $\Fix$ is a $T_\Gamma$-definable set.
\end{enumerate}
In this case, a simple $T_{\Gamma}$-formula witnesses that $\Fix$ is a definable set.
\end{thm}

\begin{proof}
The direction that (1) implies (2) follows immediately from Proposition \ref{defT}.  The other direction follows immediately from the definition of property (T) and Corollary \ref{specgapdef}.  The moreover part follows from the existence of Kazhdan sets.
\end{proof}

\begin{nrmk}
Note that, at first glance, property (T) merely states that $\Fix(\pi)$ is a $\varphi_{\Gamma}$-definable subset of $\mathcal{H}_{\pi}$ for each $\mathcal{H}_{\pi}\models T_{\Gamma}$.  However, after some simple Hilbert space manipulations, one concludes the stronger statement that $\Fix(\pi)$ is a $T_{\Gamma}$-definable set. 
\end{nrmk}

The phenomenon described in the previous remark is atypical; here is an example to show that it need not hold in general:

\begin{ex}
Let $L$ consist of a single unary relation symbol $R$ taking values in $[0,1]$ and let $T$ be the $L$-theory that states that $R$ is constant in every model, that is, $$T=\left\{\sup_{x,y}|R(x)-R(y)|=0\right\}.$$  Then in any $\mathbf{M}\models T$, $Z(R^{\mathbf{M}})$ is either empty or all of $\mathbf{M}$; in either case, it is $R$-definable.  (If $R^{\mathbf{M}}$ is constantly $\delta$, then $\frac{\delta}{2}$ works for any $\epsilon$ vacuously.)  But if $\mathbf{M}_n\models T$ is such that $R^{\mathbf{M}_n}$ is constantly $\frac{1}{n}$ and $\mathbf{M}=\prod_\u \mathbf{M}_n$ for some nonprincipal ultrafilter $\u$ on $\n$, then $R^{\mathbf{M}}$ is identically $0$, whence $Z(R^{\mathbf{M}})=\mathbf{M}$ while $\prod_\u Z(R^{\mathbf{M}_n})=\emptyset$, whence $Z(R)$ is not a $T$-definable set.
\end{ex}


\subsection{Spectral gap and ergodic theory}

If $\Gamma$ has property (T), then there are strong implications for the ergodic theory of actions of $\Gamma$.  In this subsection, we point out the model-theoretic versions of these implications.

We first recall that an action $\sigma$ of $\Gamma$ on a probability space $(X,\mathcal{B},\mu)$ is said to be \textbf{probability measure preserving} (or pmp) if each $\gamma\in \Gamma$ acts as an automorphism of the probability space.  In this case, we simply write $\Gamma\curvearrowright^{\sigma}(X,\mathcal{B},\mu)$ or even $\Gamma\curvearrowright^{\sigma} (X,\mu)$ if $\mathcal{B}$ is clear from context.

\begin{df}
Given an action $\Gamma\curvearrowright^{\sigma}(X,\mu)$, we set $$\Fix(\sigma):=\{A\in \mathcal{B} \ : \ gA=A \text{ for all }g\in G\}.$$  We say that $\sigma$ is \textbf{ergodic} if every element of $\Fix(\sigma)$ is $\mu$-null or $\mu$-conull.
\end{df}

The connection between pmp actions and unitary representations is via the following definition:

\begin{df}
Given an action $\Gamma\curvearrowright^{\sigma} (X,\mu)$, the \textbf{Koopman representation} of $\sigma$ is the unitary group representation $\pi_\sigma:\Gamma\to L^2(X,\mu)$ given by $$(\pi_\sigma(\gamma)(f))(x):=f(\sigma(\gamma^{-1})(x)).$$
\end{df}

Note that $\pi_\sigma$ is never ergodic since it always contains the constant functions.  We let $\pi_{\sigma,0}$ denote the restriction of $\pi_\sigma$ to the orthogonal complement of the constant functions.  The following lemma is standard and straightforward.

\begin{lemma}
$\sigma$ is ergodic if and only if $\pi_{\sigma,_0}$ is ergodic.
\end{lemma}

\begin{df}
Given an action $\Gamma\curvearrowright^{\sigma} (X,\mu)$, we say that $\sigma$ has \textbf{spectral gap} if the Koopman representation $\pi_\sigma:\Gamma\to L^2(X,\mu)$ has spectral gap.
\end{df}

As described in \cite[Section 16]{BBHU}, probability spaces are studied model-theoretically via the corresponding \textbf{probability algebras}, which are simply the metric structures obtained from identifying measurable sets whose symmetric difference has measure $0$.  Probability algebras form an elementary class in a natural language.  A pmp action of $\Gamma$ on a probability space induces an action of $\Gamma$ on the corresponding probability algebra and it is straightforward to verify that the class of actions of $\Gamma$ on probability algebras forms an elementary class, say $\Mod(T_{\Gamma \curvearrowright})$.  We let $\varphi_{\Gamma\curvearrowright}$ denote the $T_{\Gamma \curvearrowright}$-formula $\sum_m 2^{-m}d(\gamma_m \cdot x,x)$.  Given an action $\Gamma\curvearrowright^{\sigma}(X,\mu)$, we abuse notation and let $\Fix(\sigma)$ also denote the set of elements of the probability algebra corresponding to $(X,\mu)$ fixed by every element of $\Gamma$, which clearly coincides with the zeroset of $\varphi_{\Gamma\curvearrowright}$.

In the next proposition, we adopt the convention that if $A$ and $B$ are measurable sets in some probability space, then $a$ and $b$ denote the corresponding element of the associated probability algebra.

\begin{prop}\label{spectralaction}
Suppose that $\sigma$ has spectral gap.  Then $\Fix(\sigma)$ is a $\varphi_{\Gamma\curvearrowright}$-definable set.
\end{prop}

\begin{proof}
Fix $\epsilon>0$ and choose a finite subset $F$ of $\Gamma$ and $\delta>0$ witnessing that $\pi_\sigma$ has spectral gap.  Fix $\eta>0$ sufficiently small and suppose that $A$ is a measurable set such $\varphi_{\Gamma\curvearrowright}(a)<\eta$.  Then, for $\eta$ chosen appropriately, it follows that $1_A$ is $(F,\delta)$-invariant, so there is $f\in \Fix(\pi_{\sigma})$ such that $\|1_A-f\|_2<\epsilon$.  Let $B:=\{x \in X \ : \ |f(x)|\geq \frac{1}{2}\}$; since $f\in \Fix(\pi_{\sigma})$, we have that $b\in \Fix(\sigma)$.  Now notice that
$$\frac{1}{4}\mu(A\triangle B)\leq \int_{A\triangle B}|1_A(x)-f(x)|^2d\mu(x)\leq \|1_A-f\|^2<\epsilon^2.$$  It follows that $d(a,b)<4\epsilon^2$.
\end{proof}

\begin{nrmk}
Unfortunately, the converse of the previous proposition is false.  Indeed, suppose that $\sigma$ is a \textbf{strongly ergodic} action, meaning that every ultrapower of $\sigma$ is ergodic.  Then by Theorem \ref{ultrapowerlocaldefinability}, we have that $\Fix(\sigma)=\{0,1\}$ is a $\varphi_{\Gamma\curvearrowright}$-definable set.  However, there are examples of strongly ergodic actions that do not have spectral gap; see \cite{Schmidt}.
\end{nrmk}

The previous remark notwithstanding, the global picture is still clear:

\begin{thm}
$\Gamma$ has property (T) if and only if $\Fix$ is a $T_{\Gamma\curvearrowright}$-definable set.
\end{thm}

\begin{proof}
Suppose $\Gamma$ has property (T).  Then the proof of Proposition \ref{spectralaction} shows that $\Fix(\sigma)$ has spectral gap uniformly over all actions (depending on how the representations of $T$ have spectral gap uniformly).

The converse follows from a theorem of Connes-Schmidt-Weiss (see \cite[Section 6.3]{Tbook}, who show that $\Gamma$ has property (T) if and only if every ergodic action of $\Gamma$ is strongly ergodic; this latter condition follows from the fact that $\Fix(\sigma)$ is a $\varphi_{\Gamma\curvearrowright}$-definable set for each $\sigma$ (which is a priori weaker than $\Fix(\sigma)$ being a $T_{\Gamma\curvearrowright}$-definable set).
\end{proof}

\section{Basic von Neumann algebra theory}

The remainder of this paper concerns von Neumann algebras.  In this section, we review the necessary background material.

\subsection{Preliminaries}  For any $X\subseteq \B(\H)$, set
$$X':=\{T\in \B(\H) \ : \ TS=ST \text{ for all }S\in X\}.$$

Note that $X'$ is a unital subalgebra of $\B(\H)$ for any $X\subseteq \B(\H)$ which is moreover closed under $*$ if $X$ is closed under $*$.

\begin{df}
A \textbf{von Neumann algebra} is a unital $*$-subalgebra $M$ of $\B(\H)$ such that $M''=M$.
\end{df}

\begin{ex}
Note that $\B(\H)'=\mathbb{C}\cdot 1$, so $\B(\H)''=\B(\H)$, whence $\B(\H)$ is a von Neumann algebra.  In particular, when $\dim(\H)=n$, we see that $M_n(\mathbb{C})$ is a von Neumann algebra.
\end{ex}

Here is a much more interesting source of examples:

\begin{ex}
We let $L(\Gamma):=\lambda_\Gamma(\Gamma)''\subseteq \B(\ell^2\Gamma)$ denote the von Neumann algebra generated by $\lambda_\Gamma(\Gamma)$.
\end{ex}

When studying von Neumann algebras, two other topologies on $\B(\H)$ prove very useful:

\begin{df}
Suppose that $\H$ is a Hilbert space.
\begin{enumerate}
\item The \textbf{strong operator topology} (SOT) on $\B(\H)$ is the weakest topology making the maps $T\mapsto Tx:\B(\H)\to \H$ (for $x\in \H$) continuous.
\item The \textbf{weak operator topology} (WOT) on $\B(\H)$ is the weakest topology making the maps $T\mapsto \langle Tx,y\rangle:\B(\H)\to \H$ (for $x,y\in \H$) continuous.
\end{enumerate}
\end{df}

It is readily verified that the weak operator topology refines the strong operator topology, which in turn refines the operator norm topology.  The remarkable \textbf{bicommutant theorem} of von Neumann states that, for a unital $*$-subalgebra $M$ of $\B(H)$, one has that $M$ is a von Neumann algebra (in the above sense, that is, $M=M''$) if and only if $M$ is WOT-closed if and only if $M$ is SOT-closed.

\begin{notation}
For a von Neumann algebra $M$, we let $M_1$ denote its operator norm unit ball.
\end{notation}

\begin{df}
Suppose that $M$ is a von Neumann algebra.  A linear functional $\tr:M\to \mathbb{C}$ is called a \textbf{trace} on $M$ if it satisfies the following properties:
\begin{itemize}
\item (normalization) $\tr(1)=1$;
\item (positivity) $\tr(x^*x)\geq 0$ for all $x\in M$;
\item (faithful) $\tr(x^*x)=0$ if and only if $x=0$;
\item (normality) $\tr|M_1$ is WOT-continuous;
\item (traciality) $\tr(xy)=\tr(yx)$ for all $x,y\in M$.
\end{itemize}
A \textbf{tracial von Neumann algebra} is a pair $(M,\tr)$, where $M$ is a von Neumann algebra and $\tr$ is a trace on $M$.\footnote{One often abuses notation and simply writes $M$ for a tracial von Neumann algebra (suppressing mention of the particular trace on $M$ that is under consideration).}
\end{df}

\begin{ex}

\

\begin{enumerate}
\item $M_n(\mathbb{C})$ has a trace given by $\tr_n(x):=\frac{1}{n}\sum_{i=1}^n x_{ii}$.  However, if $\H$ is infinite-dimensional, then $\B(\H)$ does not have a trace.
\item For any group $\Gamma$, $L(\Gamma)$ admits a trace $\tr$ given by $\tr(x):=\langle x\delta_e,\delta_e\rangle$.  
\end{enumerate}
\end{ex}

Until further notice, fix a tracial von Neumann algebra $(M,\tr)$.  We can define an inner-product $\langle\cdot,\cdot \rangle_{\tr}$ on $M$ given by $\langle x,y\rangle_{\tr}:=\tr(y^*x)$ whose corresponding norm will be denoted $\|x\|_2:=\|x\|_{\tr,2}:=\sqrt{\langle x,x\rangle_{\tr}}$.  We will say that $M$ is \textbf{separable} if the metric on $M$ induced by the norm $\|\cdot\|_2$ is separable.

We let $L^2(M)$ be the Hilbert space obtained by completing the inner product space $(M,\langle\cdot,\cdot\rangle_{\tr})$.  Using the fact that $\|ab\|_2\leq \|a\|\cdot \|b\|_2$, one can readily verify that every $a\in M$ can be viewed as an element $\hat{a}$ of $\B(L^2(M))$ defined by $\hat{a}(b):=ab$ for $b\in M$ (and then extended to the completion by the above observation).  Moreover, the embedding $a\hookrightarrow \hat{a}:M\to \B(L^2(M))$ is SOT-continuous.  This representation of $M$ is referred to as the \textbf{standard representation}.\

Continuing the discussion from the previous paragraph, suppose that $N$ is a von Neumann subalgebra of $M$.  It is straightforward to see that $L^2(N)$ is then a closed subspace of $L^2(M)$ and we let $\mathbb{E}_{N}$ denote the orthogonal projection of $L^2(M)$ onto $L^2(N)$.  It is routine to verify that $\mathbb{E}_N(M)=N$.

\begin{df}
Given a von Neumann algebra $M$, we define its \textbf{center} to be $Z(M):=M'\cap M:=\{x\in M \ : \ xy=yx \text{ for all }y\in M\}$.  A von Neumann algebra with trivial center (that is, $Z(M)=\mathbb{C}\cdot 1$) is called a \textbf{factor}.
\end{df}

\begin{ex}

\

\begin{enumerate}
\item $\B(\H)$ is a factor.
\item $L(\Gamma)$ is a factor if and only if every nontrivial conjugacy class of $\Gamma$ is infinite.  (Such groups are called \textbf{ICC} groups.)
\end{enumerate}
\end{ex}

\begin{df}
A \textbf{II$_1$ factor} is an infinite-dimensional factor that admits a trace.
\end{df}

Consequently, $\B(\H)$ is never a II$_1$ factor.  On the other hand, when $\Gamma$ is a countably infinite ICC group, then $L(\Gamma)$ is a II$_1$ factor.  It is a theorem of Connes \cite{connes} that all countably infinite ICC \textit{amenable} groups (e.g. $S_\infty:=\bigcup_n S_n$) yield the same II$_1$ factor, called the \textbf{hyperfinite II$_1$ factor}, denoted $\R$.  We should also note that a II$_1$ factor admits a unique trace.

\begin{df}
Suppose that $M\subseteq \B(\H_1)$ and $N\subseteq \B(\H_2)$ are von Neumann algebras.  We define the \textbf{tensor product} of $M$ and $N$ to be the von Neumann algebra
$$M\overline{\otimes} N:=\overline{M\odot N}^{\operatorname{WOT}}\subseteq \B(\H_1\otimes \H_2),$$ where $\odot$ denotes the usual vector space tensor product and $\otimes$ denotes the usual Hilbert space tensor product.\footnote{It can be shown that this tensor product does not depend on the representations of $M$ and $N$.}
\end{df}

It is straightforward to verify that if $M$ and $N$ are tracial, then so is $M\overline{\otimes}N$.

\begin{ex}
If $\Gamma_1$ and $\Gamma_2$ are two groups, then $L(\Gamma_1)\overline{\otimes} L(\Gamma_2)\cong L(\Gamma_1\oplus \Gamma_2)$.
\end{ex}

We will also need the notion of \textbf{amalgamated free product} of tracial von Neumann algebras as introduced in \cite{Markov}.  The context is that of two tracial von Neumann algebras $M_1$ and $M_2$ with a common subalgebra $N$.  One then constructs a tracial von Neumann algebra $M_1*_N M_2$ that is generated by $M_1$ and $M_2$, which have their common copies of $N$ identified, and which are positioned as ``freely'' as possible relative to $N$.  We will only need to know one particular instance of this freeness, namely that if $b\in M_1\setminus N$ and $c\in M_2\setminus N$, then $[b,c]\not=0$.

\subsection{Tracial von Neumann algebras as metric structures}  We now briefly describe how to view tracial von Neumann algebras as metric structures; see \cite{FHS2} for the complete details.  The model-theoretic presentation can be motivated by the tracial ultraproduct construction, which we first describe.

Let $((M_i,\tr_i))_{i\in I}$ be a family of tracial von Neumann algebras and let $\u$ be an ultrafilter on $I$.  We let
$$\ell^\infty(M_i):=\left\{(x_i)\in \prod_{i\in I} M_i \ : \ \sup_{i\in I}\|x_i\|<\infty\right\}$$ and
$$c_\u(M_i):=\left\{(x_i) \in \ell^\infty(M_i) \ : \ \lim_\u \|x_i\|_2=0\right\}.$$
The \textbf{tracial ultraproduct} of the family $(M_i)$ is the quotient algebra $\prod_\u M_i:=\ell^\infty(M_i)/c_\u(M_i)$, which can be shown to be a tracial von Neumann algebra with trace $\tr_\u((x_i)^\bullet):=\lim_\u \tr_i(x_i)$.

It is important to note that use of $\|\cdot\|$ in the definition of $\ell^\infty$ versus the use of $\|\cdot\|_2$ in $c_\u$ is not a typo but rather a crucial asymmetry that ensures that the quotient algebra is once again a von Neumann algebra.

Motivated by this asymmetry, one views a tracial von Neumann algebra $M$ as a many-sorted metric structure whose sorts are the operator norm balls of $M$ (say of natural number radii), equipped with all of its $*$-algebra structure and with its trace being viewed as a distinguished predicate.  The metric on each sort is the restriction of the metric induced by the $\|\cdot\|_2$-norm.  We let $L_{\vNa}$ be the metric signature naturally associated to such a structure.  Let us temporarily call the metric structure associated to $M$ the \textbf{dissection} of $M$.

With this set-up in place, one can show that the metric ultraproduct of a family of dissections of tracial von Neumann algebras is naturally isomorphic to the dissection of the tracial ultraproduct of the family of algebras (and really one has an equivalence of categories).  With a slight abuse of terminology, we can thus say that the class of tracial von Neumann algebras is an elementary class (where formally we mean that the class of $L_{\vNa}$-structures obtained from taking dissections of tracial von Neumann algebras is an elementary class), say $\Mod(T_{\vNa})$, and in fact $T_{\vNa}$ is a universal theory.  Concrete axioms for $T_{\vNa}$ are given in \cite{FHS2}.  We will abuse notation and use $M$ both for the tracial  von Neumann algebra $M$ and its dissection.

It is well-known that the tracial ultraproduct of a family of II$_1$ factors is once again a II$_1$ factor and it is also fairly easy to check that an ultraroot of a II$_1$ factor is once again a II$_1$ factor.  It follows that the subclass of II$_1$ factors is also an elementary class, in fact an $\forall\exists$-axiomatizable class.  Once again, concrete axioms for this class are given in \cite{FHS2}.  It is this latter fact that allows one to show that an existentially closed tracial von Neumann algebra is necessarily a II$_1$ factor.

\subsection{Property Gamma and the McDuff property}

There are a couple of properties that a II$_1$ factor may or may not have that will become relevant later in this paper.  The first was introduced by Murray and von Neumann in \cite{murrayvn}.  We will need the notations $[x,y]:=xy-yx$ and $\U(M):=\{x\in M\ : \ uu^*=u^*u=1\}$.

\begin{df}
We say that a II$_1$ factor $M$ has \textbf{property Gamma} if, for any finite $F\subseteq M$ and any $\epsilon>0$, there is $u\in \U(M)$ with $\tr(u)=0$ and such that $\max_{x\in F}\|[x,u]\|_2<\epsilon$.
\end{df}

The point of introducing property Gamma was that it allowed Murray and von Neumann to distinguish between $\R$ and $L(\mathbb{F}_2)$.  Indeed, they proved that any unitary $u\in \U(L(\mathbb{F}_2))$ that almost commutes with the unitaries associated with the generators of $\mathbb{F}_2$ is very close to the center of $L(\mathbb{F}_2)$, that is, is close to $\mathbb{C}$; since the unitaries in $\mathbb{C}$ cannot have trace close to $0$, this shows that $L(\mathbb{F}_2)$ does not have property Gamma.  Combined with the easy observation that $\R$ does have property Gamma, they were able to conclude that $\R\not\cong L(\mathbb{F}_2)$.  In \cite[3.2.2]{FHS3}, the authors showed that property Gamma is in fact an axiomatizable property of separable II$_1$ factors, whence one can conclude that $\R\not\equiv L(\mathbb{F}_2)$.

The other property that will become relevant is the following:

\begin{df}[McDuff \cite{mcduff}]
A separable II$_1$ factor $M$ is said to be \textbf{McDuff} if $M\overline{\otimes} \R\cong M$.  
\end{df}

From the presentation $\R=\overline{\bigotimes} M_{2^n}(\mathbb{C})$, it is relatively straightforward to see that $\R$ is McDuff.  Consequently, $M\overline{\otimes} \R$ is McDuff for any separable II$_1$ factor $M$.  Following Popa \cite{popasmcduff}, when $M$ is non-Gamma, we call $M\overline{\otimes} \R$ \textbf{strongly McDuff}.

In Remark \ref{McDuffimpliesGamma}, we will mention that McDuff implies property Gamma.  Dixmier and Lance \cite{DL} gave an example of a separable II$_1$ factor $M$ that has property Gamma but is not McDuff.  This $M$ provided the third isomorphism class of separable II$_1$ factors.

In \cite[Proposition 3.9]{FHS3}, the authors show that McDuffness is also axiomatizable for separable II$_1$ factors.  Consequently, the $M$ from the previous paragraph also represents a third elementary equivalence class.

\section{Spectral gap subalgebras}

\subsection{Introducing spectral gap for subalgebras}  Until further notice, suppose that $M$ is a separable II$_1$ factor and $N\subseteq M$ is a subalgebra.  For the sake of readability, all $L_{\vNa}$-formulae appearing below will be assumed to have their free variable ranging over the sort for the operator norm unit ball.

Although the presentation in Section 3 was in terms of countable groups, one can also define what it means for a unitary representation of an arbitrary (not necessarily countable) group to have spectral gap (as is done in \cite{AP}).  In particular, it makes sense to speak of a unitary representation of $U(N)$ having spectral gap.

\begin{df}
We say that $N$ \textbf{has spectral gap in} $M$ if the unitary representation $u\mapsto uxu^*:\U(N)\to \U(L^2(M))$ of $U(N)$ has spectral gap.
\end{df}

\begin{ex}
Suppose that $\Gamma$ has property (T).  Then $L(\Gamma)$ has spectral gap in $M$ for any $M$ containing $L(\Gamma)$.  Indeed, fix $\epsilon>0$ and take $F\subseteq \Gamma$ finite and $\delta>0$ witnessing that $\Gamma$ has property (T).  For $\gamma\in \Gamma$, set $u_\gamma:=\lambda_\Gamma(\gamma)$.  It follows that, for any $a\in L^2(M)$, if $\max_{\gamma\in F}\|u_\gamma a u_\gamma^*-a\|<\delta$, then there is $b\in L^2(M)$ with $u_\gamma bu_\gamma^*=b$ for all $\gamma\in \Gamma$ and with $\|a-b\|_2\leq \epsilon$.  It just remains to observe that such $b$ then commutes with all of $L(\Gamma)$ and thus with all of $\U(L(\Gamma))$.
\end{ex}

\begin{nrmk}\label{Tremark}
The preceding example can be generalized.  Indeed, there is a notion of a II$_1$ factor having property (T) (examples of which include $L(\Gamma)$ for $\Gamma$ an ICC property (T) group) and such II$_1$ factors will have spectral gap in any extension.  In a project in progress with Bradd Hart and Thomas Sinclair, we generalize the results in Section 3 by showing that a II$_1$ factor $M$ has property (T) if and only if the set of central vectors is a definable set relative to theory of $M$-$M$ bimodules.
\end{nrmk}


Let us recast the notion of spectral gap subalgebra in more concrete terms.  Indeed, we have that $N$ has spectral gap in $M$ if, for all $\epsilon>0$, there are $u_1,\ldots,u_n\in \U(M)$ and $\delta>0$ such that, for all $x\in M$, 
$$\|[x,u_i]\|_2\leq \delta\|x\|_2\Rightarrow \|x-\mathbb{E}_{N'\cap M}(x)\|_2\leq \epsilon\|x\|_2.$$

By weakening the previous statement by asking that $x$ above only range over $M_1$, one obtains a useful weakening of the notion of spectral gap subalgebra:

\begin{df}
$N$ has \textbf{weak spectral gap} (or \textbf{w-spectral gap}) in $M$ if for all $\epsilon>0$, there are $u_1,\ldots,u_n\in \U(M)$ and $\delta>0$ such that, for all $x\in M_1$, 
$$\|[x,u_i]\|_2\leq \delta\|x\|_2\Rightarrow \|x-\mathbb{E}_{N'\cap M}(x)\|_2\leq \epsilon\|x\|_2.$$
\end{df}

We leave the following lemma as an exercise to the reader:

\begin{lemma}\label{ultraspecgap}

\

\begin{enumerate}
\item $N$ has spectral gap in $M$ if and only if $N'\cap L^2(M)^\u=L^2(N'\cap M)^\u$. 
\item $N$ has w-spectral gap in $M$ if and only if $N'\cap M^\u=(N'\cap M)^\u$.
\end{enumerate}
\end{lemma}

In general, spectral gap and weak spectral gap are different notions (see, for example, \cite[Remark 2.2]{popa09}).  There is an important case in which they coincide:

\begin{fact}[Connes \cite{connes}]\label{connesfact}
Suppose that $N$ is a II$_1$ factor.  Then the following are equivalent:
\begin{enumerate}
\item $N$ has spectral gap in $N$;
\item $N$ has w-spectral gap in $N$ (i.e. $N'\cap N^\u=\mathbb{C}\cdot 1$);
\item $N$ does not have property Gamma.
\end{enumerate}
Moreover, if these equivalent conditions hold, then $N$ has spectral gap in $N\overline{\otimes} S$ for any tracial von Neumann algebra $S$.
\end{fact}

\begin{nrmk}
As a corollary of the previous fact, if $N$ is a II$_1$ factor with spectral gap in some extension $M$, then $N$ does not have property Gamma.
\end{nrmk}

\begin{nrmk}\label{McDuffimpliesGamma}
In her paper \cite{mcduff}, McDuff shows that a separable II$_1$ factor $M$ is McDuff if and only if $M'\cap M^\u$ is not abelian.  Combined with Fact \ref{connesfact}, we see that McDuff implies property Gamma.
\end{nrmk}

\subsection{Spectral gap and definability}  Let $\{u_n\}$ be an enumeration of a countable dense subset of $\U(N)$ and let $\varphi_N(x):=\sum_n 2^{-n}\|[x,u_n]\|_2$, a formula in $M$ over $N$.  Note that $Z(\varphi_N)=N'\cap M_1$.  The following theorem is almost immediate from the definition:  

\begin{thm}\label{defcomm1}
$N$ has w-spectral gap in $M$ if and only if $N'\cap M_1$ is a $\varphi_N$-definable subset of $M_1$.
\end{thm}

\begin{nrmk}\label{unfortunate2}
As in the case of spectral gap for unitary representations, once again we cannot replace ``$\varphi_N$-definable'' with ``definable'' in the previous theorem.  For instance, if $N=M$, then $M'\cap M_1=\mathbb{S}^1$, which is a definable subset of $M$ (as it is compact), but as we just saw in the last subsection, $M$ has w-spectral in itself if and only if $M$ does not have property Gamma. 
\end{nrmk}

For the sake of sanity, let us say that a subalgebra $Q$ of $M$ is definable if $Q\cap M_1$ is a definable subset of $M_1$.

Na\"ively speaking, it seems that $N'\cap M$ should always be a definable subalgebra of $M$.  Indeed, $N'\cap M_1$ is the zeroset of $\sup_{y\in N_1}\|[x,y]\|_2$.  There are two issues with this train of thought.  First, the aforementioned expression is only a formula in $M$ if $N$ is a definable subalgebra of $M$.  Secondly, zerosets need not be definable sets.  It turns out that the second issue is not really an issue at all.

\begin{fact}[See Lemma 3.6.5(ii) in \cite{SS}]\label{dixmier}
For any $x\in M_1$, we have $$d(x,N_1)\leq \sup_{y\in N_1}\|[x,y]\|_2.$$
\end{fact}

Theorem \ref{localdefinabilitytheorem} and the previous fact immediately imply the following:

\begin{cor}\label{defcomm2}
Suppose that $N$ is a definable subalgebra of $M$.  Then $N'\cap M$ is a definable subalgebra of $M$.
\end{cor}

\begin{ex}
Suppose that $N$ is a non-Gamma II$_1$ factor and $S$ is any tracial von Neumann algebra.  Then by Fact \ref{connesfact}, $N$ has w-spectral gap in $N\overline{\otimes}S$, whence $S=N'\cap (N\overline{\otimes} S)$ is a definable subalgebra of $N\overline{\otimes} S$.  Moreover, by Corollary \ref{defcomm2}, we have that $N=S'\cap (N\overline{\otimes} S)$ is also a definable subalgebra of $N\otimes S$.
\end{ex}

Combining Theorem \ref{defcomm1} with Corollary \ref{defcomm2} yields:

\begin{cor}\label{defcomm3}
Suppose that $N$ has w-spectral gap in $M$.  Then $(N'\cap M)'\cap M$ is a definable subalgebra of $M$.
\end{cor}

\subsection{Relative bicommutants and e.c. II$_1$ factors}  Recalling von Neumann's double commutant theorem, one might see the above relative bicommutant and guess that $(N'\cap M)'\cap M$ should always coincide with $N$.  However, this is often not the case.  Indeed, given a II$_1$ factor $M$, one can always find a proper \emph{irreducible} subfactor $N$ in the sense that $N'\cap M=\mathbb{C}\cdot 1$, in which case $(N'\cap M)'\cap M=M$.  (For $M\not=\R$, this is due to Popa \cite[Corollary 4.1]{PopaInv}\footnote{We thank Stefaan Vaes for pointing us to this reference.}; for $M=\R$, this follows from the work of Jones in \cite{jones}.)

In connection with the above discussion, the following exercise in Hodges' book \emph{Building Models by Games} \cite[Exercise 3.3.2(b)]{hodges} proved inspiring to the current discussion:

\begin{fact}
Suppose that $G$ is an existentially closed group and $a\in G$.  Then $C_G(C_G(a))=\langle a\rangle$.\footnote{Here, for $X\subseteq G$, $C_G(X):=\{b\in G \ : \ ab=ba \text{ for all }a\in X\}$, $C_G(a):=C_G(\{a\})$, and $\langle a \rangle$ denotes the subgroup of $G$ generated by $a$.}
\end{fact}

As we just pointed out, the na\"ive von Neumann analog of the previous fact is not true.  However, the von Neumann analog does hold with a spectral gap hypothesis:

\begin{prop}\label{ecbicomm}
Suppose that $M$ is an e.c.\ II$_1$ factor and $N$ is a w-spectral gap subalgebra of $M$.  Then $N$ satisfies the bicommutant condition $$(N'\cap M)'\cap M=N.$$
\end{prop}

\begin{proof}
Suppose, towards a contradiction, that $b\in (N'\cap M)'\cap M$ but $b\notin N$.  Let $Q:=M*_N (N\overline{\otimes}L(\Z))$.  Since $M\subseteq Q$ and $M$ is e.c., there is an embedding $i:Q\to M^\u$ such that $i$ restricts to the diagonal embedding on $M$.  Let $c\in Q$ be the canonical unitary of $L(\mathbb{Z})$.  Then $i(c)\in N'\cap M^\u=(N'\cap M)^\u$, so we can write $i(c)=(c_n)^\bullet$ with each $c_n\in N'\cap M$.  By choice of $b$, we have $[b,c_n]=0$ for all $n$, whence $[i(b),i(c)]=0$ and hence $[b,c]=0$, contradicting the fact that $b\notin N$.
\end{proof}

%

\begin{cor}
Suppose that $M$ is an e.c. II$_1$ factor and $N$ has w-spectral gap in $M$.  Then $N$ is a definable subalgebra of $M$.
\end{cor}


\begin{cor}
Suppose that $M$ is an e.c.\ II$_1$ factor and $N$ is a property (T) subfactor.\footnote{Refer back to Remark \ref{Tremark} for a discussion on property (T) factors.}  Then $N$ is a definable subalgebra of $M$.
\end{cor} 

The above discussion can be used to give a new proof of the following, fact previously established by the author, Bradd Hart, and Thomas Sinclair in \cite{nomodcomp}.

\begin{cor}
The theory of II$_1$ factors does not have a model companion.
\end{cor}

\begin{proof}
Let $\Gamma$ be an infinite, ICC group with property (T) (e.g. $\Gamma:=\operatorname{SL}_3(\mathbb{Z})$) and set $N:=L(\Gamma)$.  Let $M$ be an e.c.\ II$_1$ factor containing $N$.  We seek to find an elementary extension $\tilde{M}$ of $M$ that is not e.c.  \

Set $\psi_N(x):=\sup_{y\in N'\cap M}\|[x,y]\|_2$, so $N=(N'\cap M)'\cap M$ is $\psi_N$-definable.  Since $N$ is infinite-dimensional, by compactness there is an elementary extension $\tilde{M}$ of $M$ such that $Z(\psi_N^{\tilde{M}})$ is a proper extension of $N$.  Now note that $\psi_N^{\tilde{M}}$ defines $(N'\cap \tilde{M})'\cap \tilde{M}$.  It follows that $(N'\cap \tilde{M})'\cap\tilde{M}\not=N$, whence $\tilde{M}$ is not e.c.
\end{proof}

It is worth pointing out that the previous proof uses far less ``technology'' than the original proof from \cite{nomodcomp}.  Indeed, the above proof simply uses the existence of infinite ICC groups with property (T) while the proof from \cite{nomodcomp} uses some deep results of Bekka \cite{Bekka} and Brown \cite{Brown}.

%

\subsection{Open questions}  We end this section by mentioning some open problems where spectral gap might play a role:

\begin{question}
Are any two e.c. II$_1$ factors elementarily equivalent?
\end{question}

Here is a quick explanation for why the ideas presented in this section might be relevant to the previous question.  For simplicity, suppose that $\R$ is e.c.  (As mentioned in \cite{FGHS}, $\R$ is e.c. if and only if the famous \emph{Connes Embedding Problem} has a positive solution.)  Let $N$ be a property $(T)$ factor (so, in particular, does not have property Gamma).  Let $M$ be an e.c. factor containing $N$.  The hope would be to show that this $M$ could not be elementarily equivalent to $\R$.  Indeed, $N$ is a definable subfactor of $M$ and the idea would be to see if one could use the assumption that $M\equiv \R$ to show that $\R$ must also have a definable, non-Gamma subfactor, yielding a contradiction (as all subfactors of $\R$ are hyperfinite).
%

If the strategy described in the previous paragraph worked, it could probably also be used to yield a positive answer to the following:

\begin{question}
If $M$ is a strongly McDuff II$_1$ factor, is it true that $M\not\equiv \R$?
\end{question}


A related question:

\begin{question}
Can an e.c. II$_1$ factor ever be strongly McDuff?
\end{question}

Temporarily, call a non-Gamma factor $N$ \textbf{bc-good} if it contains a w-spectral gap subfactor $\tilde{N}$ such that $(\tilde{N}'\cap N)'\cap N\not=\tilde{N}$.  Corollary \ref{defcomm3} immediately implies:

\begin{lemma}
Suppose that $N$ is a non-Gamma II$_1$ factor with a w-spectral gap subfactor $\tilde{N}$ that is not definable.  Then $N$ is bc-good.
\end{lemma}
%

Here is the relevance of bc-good factors in connection with the last question:
\begin{cor}
If $N$ is a bc-good non-Gamma factor, then $N\overline{\otimes} \R$ is \emph{not} e.c.
\end{cor}

\begin{proof}
Let $\tilde{N}$ be as in the definition of bc-good.  Then $\tilde{N}$ has w-spectral gap in $N\overline{\otimes} \R$ (see \cite[Corollary in the Appendix]{ioana}) but $$(\tilde{N}'\cap(N\overline{\otimes} \R))'\cap (N\overline{\otimes} \R)=(\tilde{N}'\cap N)'\cap N\not=\tilde{N},$$ whence, by Proposition \ref{ecbicomm}, $N\overline{\otimes} \R$ is not e.c.
\end{proof}

Consequently, it is of interest to investigate whether or not all non-Gamma factors are bc-good.
%

\section{Continuum many theories of II$_1$ factors}

\subsection{The history and the main theorem} 

The progress in exhibiting many nonisomorphic separable II$_1$ factors was very slow.  As mentioned earlier, the first example of two nonisomorphic separable II$_1$ factors was given by Murray and von Neumann, where they used property Gamma to distinguish $\R$ from $L(\mathbb{F}_2)$.  The third isomorphim class was given by Dixmier and Lance, where they found an example of a separable II$_1$ factor that had property Gamma but was not McDuff.

Slowly, more isomorphism classes were discovered but it remained until McDuff's seminal works in \cite{MD1} and \cite{MD2} to exhibit infinitely many isomorphism classes.  The latter work exhibited a family $(M_{\balpha})_{\balpha\in 2^\omega}$ of pairwise nonisomorphic separable II$_1$ factors.  We will explain the construction of this family below.
 
The progress in exhibiting non-elementarily equivalent II$_1$ factors was equally as slow.  As mentioned earlier, property Gamma and the McDuff property were shown to be elementary properties in \cite{FHS3}, so the first three pairwise nonisomorphic separable II$_1$ factors are also pairwise non-elementarily equivalent.  A fourth class was discovered by the current author and Bradd Hart in 2015 (but was not published until \cite{braddisaac}) and it remained open to show that there were infinitely many elementary equivalence classes until the appearance of \cite{BCI}, where the following was shown:

\begin{fact}
Let $(M_{\balpha})_{\balpha\in 2^\omega}$ be McDuff's family of pairwise nonisomorphic separable II$_1$ factors.  Then for any $\balpha\not=\bbeta$ and ultrafilters $\u$ and $\mathcal{V}$ on any sets, we have that $M_{\balpha}^{\u}\not\cong M_{\bbeta}^{\mathcal{V}}$.
\end{fact}

Using either the continuous version of the Keisler-Shelah theorem or a Continuum Hypothesis/absoluteness argument, the following model-theoretic corollary is immediate:

\begin{cor}\label{eecor}
Under the assumptions of the previous fact, we have that $M_{\balpha}\not\equiv M_{\bbeta}$.
\end{cor} 

The previous corollary notwithstanding, it at first proved difficult to find explicit sentences distinguishing the McDuff examples.  In \cite{braddisaac}, with Bradd Hart we used Ehrenfeucht-Fra\"isse games together with a careful reading of \cite{BCI} to at least give an upper bound to the quantifier-complexity of sentences distinguishing the McDuff examples.  Refining the ideas in \cite{braddisaac}, together with Hart and Henry Towsner, we were finally able to provide explicit sentences distinguishing the McDuff examples in \cite{braddisaachenry}.  It is the goal of this section to give a rough overview of the ideas involved in this latter work and to highlight the role of spectral gap and definability.  

Before stating the main theorem, let us introduce some notation.  Let $\Gamma$ be a countable group.  For $i\geq 1$, let $\Gamma_i$ denote an isomorphic copy of $\Gamma$ and let $\Lambda_i$ denote an isomorphic copy of $\mathbb{Z}$.  Let $\tilde{\Gamma}:=\bigoplus_{i\geq 1}\Gamma_i$.  If $S_\infty$ denotes the group of permutations of $\n$ with finite support, then there is a natural action of $S_\infty$ on $\tilde{\Gamma}$ (given by permutation of indices), whence we may consider the semidirect product $\tilde{\Gamma}\rtimes S_\infty$.  Given these conventions, we can now define two new groups:

$$T_0(\Gamma):=\langle \tilde{\Gamma}, (\Lambda_i)_{i\geq 1} \ | \ [\Gamma_i,\Lambda_j]=0 \text{ for }i\geq j\rangle$$ and

$$T_1(\Gamma):=\langle \tilde{\Gamma}\rtimes S_\infty, (\Lambda_i)_{i\geq 1} \ | \ [\Gamma_i,\Lambda_j]=0 \text{ for }i\geq j\rangle.$$

Note that if $\Delta$ is a subgroup of $\Gamma$ and $\alpha\in \{0,1\}$, then $T_\alpha(\Delta)$ is a subgroup of $T_\alpha(\Gamma)$.  Given a sequence $\balpha\in 2^{\leq \omega}$, we define a group $K_{\balpha}(\Gamma)$ as follows:
\begin{enumerate}
\item $K_{\balpha}(\Gamma):=\Gamma$ if $\balpha=\emptyset$;
\item $K_{\balpha}(\Gamma):=(T_{\alpha_0}\circ T_{\alpha_1}\circ \cdots T_{\alpha_{n-1}})(\Gamma)$ if $\balpha\in 2^n$;
\item $K_{\balpha}$ is the inductive limit of $(K_{\balpha|n})_n$ if $\balpha\in 2^\omega$.
\end{enumerate}

We then set $M_{\balpha}(\Gamma):=L(T_{\balpha}(\Gamma))$; if $\balpha=(0)$ or $(1)$, we simply write $M_0(\Gamma)$ or $M_1(\Gamma)$.  We also set $M_{\balpha}:=M_{\balpha}(\mathbb{F}_2)$; these are the McDuff examples referred to above.

Here is the main theorem from \cite{braddisaachenry}:

\begin{thm}\label{mainBCIagain}
For each nonamenable ICC group $\Gamma$, there is an integer $m(\Gamma)$ and a sequence $(c_n(\Gamma))$ of positive real numbers such that, for any $n,t\in \n$ with $t\geq 1$ and any $\balpha\in 2^n$, we have:
\[
\begin{array}{lr}
\theta_{m,n}^{L(T_{\balpha}(\Gamma))^{\overline{\otimes} t}}=0 \text{ for all }m\geq 1 & \text{ if }\balpha(n-1)=1;\\
\theta_{m(\Gamma),n}^{L(T_{\balpha}(\Gamma))^{\overline{\otimes} t}}\geq c_n(\Gamma) &\text{ if }\balpha(n-1)=0.
\end{array}
\]
\end{thm}

\noindent We then have the following precise form of Corollary \ref{eecor}:

\begin{cor}
Suppose that $\balpha,\bbeta\in 2^\omega$ are such that $\balpha|n-1=\bbeta|n-1$, $\balpha(n)=1$, $\bbeta(n)=0$.  Write $\bbeta=(\bbeta|n+1) \concat \bbeta^*$. Set $m:=m(T_{\bbeta^*}(\mathbb{F}_2))$ and $c:=c(T_{{\bbeta}^*}(\mathbb{F}_2))$.  Then $\theta_{m,n}^{M_{\balpha}}=0$ and $\theta_{m,n}^{M_{\bbeta}}\geq c$.
\end{cor}  

In the rest of this section, we explain the main ideas behind the proof of Theorem \ref{mainBCIagain}.

\subsection{The base case}  We start by describing how to find $\theta_{m,0}$.  First, we will need the following fact:

\begin{fact}
Let $\varphi_{\operatorname{unitary}}(U)$ be the formula $\max(\|UU^*-1\|_2,\|U^*U-1\|_2)$.  Then $Z(\varphi_{\operatorname{unitary}})$ is a $T_{\vNa}$-definable set.
\end{fact}

In light of the previous fact, in the sequel, we may consider quantifiers over unitaries.  When doing so, we will use the letters $U$ and $V$ (perhaps with subscripts) to denote variables ranging only over unitaries.

We now define the following formulae:
\begin{itemize}
\item $\chi(X,U_1,U_2):=100(\|[X,U_1]\|_2+\|[X,U_2]\|_2)$.  
\item For $m\geq 1$, set $\psi_m(V_1,V_2)$ to be
$$\sup_{\vec X,\vec Y}\left(\left( \inf_U \max_{1\leq i,j\leq m}\|[UX_iU^*,Y_j]\|_2\dotminus 2\max_{1\leq i\leq m}\sqrt{\chi(X_i,V_1,V_2)}\right)\right).$$
\item Set $\theta_{m,0}:=\inf_{V_1,V_2}\psi_m$.
\end{itemize}

We can now distinguish the base case.  First, using lemmas from \cite{BCI}, one establishes:

\begin{fact}
Suppose that $m,t\geq 1$.  Then $\theta_{m,0}^{M_1(\Gamma)^{\overline{\otimes} t}}=0$.
\end{fact}

Spectral gap comes into play in the following:

\begin{fact}\label{nonamenable}
Suppose that $\Gamma$ is not amenable.  Then there  is $m=m(\Gamma)$ and $c=c(\Gamma)$ such that, for any von Neumann algebra $Q$ and any $M_0(\Gamma)\subseteq M\subseteq M_0(\Gamma)\overline{\otimes} Q$, we have $\theta_{m,0}^M\geq c$.
\end{fact}

While the proof of this is quite technical, we mention that the key fact underlying the proof is the statement from Example \ref{amenablespec}, namely that $\Gamma$ is non-amenable if and only if $\lambda_\Gamma$ has spectral gap.  Thus, the $m$ and the $c$ come from Proposition \ref{spectgap}(1).

\subsection{A digression on good unitaries and generalized McDuff factors}   Before explaining how the inductive step works, we need to introduce some key definitions.  First, some:

\begin{notation}
Let $M$ be a von Neumann algebra and $\vec a$, $\vec b$ be sequences from $M$.
\begin{itemize}
\item We write $C_M(\vec a):=\{b\in M \ : \ [b,a_i]=0 \text{ for all }i\}$.  We usually just write $C(\vec a)$ if $M$ is clear from context.
\item We write $\vec a\leq \vec b$ to mean $C(\vec b)\subseteq C(\vec a)$.
\end{itemize}
\end{notation}

\begin{df}
Suppose that $M$ is a II$_1$ factor.  A pair $(u_1,u_2)\in \U(M)^2$ is called a \textbf{pair of good unitaries} if, for all $a\in M$, we have $d(a,C(u_1,u_2))\leq \sqrt{\chi(a,u_1,u_2)}$. 
\end{df}

In other words, a pair of unitaries $(u_1,u_2)$ is good if $C(u,v)$ is a $\sqrt{\chi}$-definable subset of $M$ with ``modulus of continuity'' the identity function.  This is also equivalent to saying that the algebra generated by $u$ and $v$ has spectral gap in $M$ in a precise numerical way.  The following two lemmas are clear:

\begin{lemma}
There is a formula $\varphi_{\good}(\vec U)$ such that, in \emph{$\omega$-saturated} II$_1$ factors (e.g. ultraproducts) $M$, $Z(\varphi_{\good}^M)$ is the set of pairs of good unitaries in $M$. 
\end{lemma}

\begin{lemma}
There is a formula $\psi(\vec X,\vec U)$ such that, for all II$_1$ factors $M$, all $\vec a\in M$, and all pairs of good unitaries $\vec u$ from $M$, we have $\vec a\leq \vec u$ if and only if $\psi(\vec a,\vec u)^M=0$.
\end{lemma}

We come to our other key definition.

\begin{df}
A \textbf{generalized McDuff ultraproduct for $\Gamma$ and $\balpha$} is one of the form
$$\prod_\u M_{\balpha}(\Gamma)^{\overline{\otimes} t_s}.$$
\end{df}

The following fact explains the importance of generalized McDuff ultraproducts: 

\begin{fact}[\cite{braddisaachenry} elaborating on results from \cite{BCI}]
Suppose that $\Gamma$ is an ICC group, $\balpha\in 2^{<\omega}$ and $M$ is a generalized McDuff ultraproduct for $\Gamma$ and $\balpha$.  Write $\balpha=\alpha_0\balpha^{\#}$.  Then for any pair of good unitaries $\vec u$ and any countable sequence $\vec a$ from $M$, there is a pair of good unitaries $\vec v$ from $M$ such that $\vec v>(\vec a,\vec u)$ and $C(\vec v)'\cap C(\vec u)$ is a generalized McDuff ultraproduct corresponding to $\Gamma$ and $\balpha^{\#}$.
\end{fact}

\subsection{The inductive step}  The preceding fact illustrates how we should proceed to find $\theta_{m,1}$.  Indeed, suppose that $\balpha=(\alpha_0,1)$.  Fix $t\geq 1$ and let $M:=\prod_\u M_{\balpha}^{\overline{\otimes} t}$.  Let $\vec u$ be a nontrivial pair of good unitaries in $M$.  Given any $a\in M$, there is $\vec v>(a,\vec u)$.  Since $C(\vec v)'\cap C(\vec u)$ is generalized McDuff for $\Gamma$ and $(0)$, Fact \ref{nonamenable} (and the \L os theorem) implies that $\theta_{m,0}^{C(\vec v)'\cap C(\vec u)}=0$.  In order to extract a genuine sentence witnessing this phenomenon, we need:  

\begin{fact}[\cite{braddisaachenry}]
For any sentence $\theta$ in prenex normal form, there is a formula $\tilde{\theta}(\vec U,\vec V)$ such that, for any II$_1$ factor $M$ and pairs of good unitaries $\vec u$ and $\vec v$ from $M$ with $C(\vec v)\subseteq C(\vec u)$, we have 
$$\tilde{\theta}(\vec u,\vec v)^M=\theta^{C(\vec v)'\cap C(\vec u)}.$$
\end{fact}

\begin{proof}[Proof Sketch]
We sketch a softer proof here than the one given in \cite{braddisaachenry}.  First let $L_1$ denote the extension of the language of tracial von Neumann algebras obtained by adding two new constants $c_{u_1}$ and $c_{u_2}$.  We let $T_1$ denote the $L_1$-theory extending the theory of tracial von Neumann algebras stating that the interpretations of $c_{u_1}$ and $c_{u_2}$ are good unitaries.  It is then clear that the $T_1$-functor mapping a model $(M,u_1,u_2)$ of $T_1$ to $C(u_1,u_2)$ is a $T_1$-definable set and the witnessing $T_1$-formula is a simple $L_1$-formula. 

Next, let $L_2$ denote the extension of $L_1$ obtained by adding two further constant symbols $c_{v_1}$ and $c_{v_2}$.  Let $T_2$ be the $L_2$-theory obtained by adding to $T_1$ axioms stating that the interpretations of $c_{v_1}$ and $c_{v_2}$ are also good unitaries and that $C(v_1,v_2)$ is contained in $C(u_1,u_2)$.  (To do this, one needs the result from the previous paragraph, namely that $C(v_1,v_2)$ is a definable set.)  Observe now that the $T_2$-functor mapping a model $(M,u_1,u_2,v_1,v_2)$ of $T_2$ to $C(v_1,v_2)'\cap C(u_1,u_2)$ is a $T_2$-definable set with witnessing formula a simple  $L_2$-formula.  Indeed, let $\varphi(x):=\max(d(x,C(u_1,u_2)),\sup_{y\in C(v_1,v_2)}\|[x,y\|_2)$.  If $\varphi(x)$ is small, then by the first paragraph, $x$ is near $x'\in C(u_1,u_2)$ for which $\sup_{y\in C(v_1,v_2)}\|[x',y]\|_2$ is still small.  But then by Fact \ref{dixmier} (applied to $C(u_1,u_2)$), we see that $x'$ is near $x''\in C(v_1,v_2)'\cap C(u_1,v_1)$.

Now given any sentnence $\theta$ (in the original language of von Neumann algebras) that is in prenex normal form, it is straightforward to construct an $L_2$-sentence $\theta'$ such that, in models $(M,u_1,u_2,v_1,v_2)$ of $T_2$, we have $(\theta')^M=\theta^{C(v_1,v_2)'\cap C(u_1,u_2)}$.  The desired formula $\tilde{\theta}$ is obtained by replacing the new constants by free variables.
\end{proof}

Given the previous fact, we  can consider the sentence $\theta_{m,1}$ given by
$$\inf_{\vec U}\max(\varphi_{\good}(\vec U),\sup_A \inf_{\vec V}\max(\varphi_{\good}(\vec V),\psi(A,\vec U, \vec V),\tilde{\theta}_{m,0}(\vec U,\vec V))).$$

Our above discussion shows the following:

\begin{prop}
Suppose that $\Gamma$ is any ICC group and $\balpha=(\alpha_0,1)$.  Then for any $t\geq 1$, we have $\theta_{m,1}^{M_{\balpha}(\Gamma)^{\overline{\otimes} t}}=0$.
\end{prop}

Contrast this with the following:

\begin{prop}
Suppose that $\Gamma$ is a non-amenable ICC group.  Then there is a constant $c_1(\Gamma)$ such that, for any $t\geq 1$ and any $\bbeta\in 2^1$ with $\bbeta(1)=0$, we have $\theta_{m,1}^{M_{\bbeta}(\Gamma)^{\overline{\otimes} t}}\geq c_1(\Gamma).$
\end{prop}

\begin{proof}
Suppose, towards a contradiction, that no such $c_1(\Gamma)$ exists.  Then there is $\bbeta\in 2^1$ such that, for any $n$, there is $t_n\geq 1$ such that $\theta_{m,1}^{M_{\bbeta}(\Gamma)^{\overline{\otimes} t_n}}<\frac{1}{n}$.  Let $M:=\prod_\u M_{\bbeta}(\Gamma)^{\overline{\otimes} t_n}$, a generalized McDuff factor corresponding to $\Gamma$ and $\bbeta$.  Let $\vec u$ be a pair of good unitaries in $M$ witnessing the outermost infimum.  Take $a>\vec u$ and take good unitaries $\vec v>(a,\vec u)$ witnessing the next infimum, that is, such that $\widetilde{\theta_{m,0}}(\vec u,\vec v)=0$.  In other words, $\theta_{m,0}^{C(v_1,v_2)'\cap C(u_1,u_2)}=0$.  But $C(v_1,v_2)'\cap C(u_1,u_2)$ is a generalized McDuff factor corresponding to $\Gamma$ and $(0)$, contradicting Fact \ref{nonamenable}.
\end{proof}

One proves Theorem \ref{mainBCIagain} by iterating the above procedure.

\end{document}